\documentclass[12pt]{amsart}%
\usepackage{amsfonts}
\usepackage{amsmath}
\usepackage{amssymb}
\usepackage{color}
\usepackage{graphicx}%
\usepackage[pagebackref]{hyperref}
\makeatletter
\@addtoreset{equation}{section}
\makeatother

\newtheorem{theorem}{Theorem}[section]

\newtheorem{lemma}[theorem]{Lemma}

\newtheorem{proposition}[theorem]{Proposition}
\newtheorem{remark}[theorem]{Remark}

\makeatletter
\def\@makefnmark{}
\makeatother

\begin{document}
\title[Bubbling phenomenon for semilinear equations of exponential growth]{Bubbling phenomenon for semilinear Neumann elliptic equations of critical exponential growth}

\author{Lu Chen, Guozhen Lu and Caifeng Zhang}
\address{School of Mathematics and Statistics, Beijing Institute of Technology, Beijing 100081, P. R. China}
\email{chenlu5818804@163.com}

\address{Department of Mathematics\\
University of Connecticut\\
Storrs, CT 06269, USA}
\email{guozhen.lu@uconn.edu}

\address{Department of Applied Mathematics, School of Mathematics and Physics, University of Science and
Technology of Beijing, Beijing 100083, P. R. China}
\email{zhangcaifeng1991@mail.bnu.edu.cn }

\thanks{The first author was partly supported by a grant from the NNSF of China (No.12271027). The second author was partly supported by a Simons Collaboration Grant from the Simons Foundation. The third author was partly supported by a
grant from the NNSF of China (No. 12001038).}

\begin{abstract}
In the past few decades, much attention has been paid to the bubbling problem for semilinear Neumann elliptic equation with the critical and subcritical polynomial nonlinearity, much less is known if the polynomial nonlinearity is replaced by the exponential nonlinearity. In this paper, we consider the following semilinear Neumann elliptic problem with the Trudinger-Moser exponential growth:
\begin{equation*}\begin{cases}
-d\Delta u_d+u_d=u_d(e^{u^2_{d}}-1)\ \ \mbox{in}\ \Omega,\\
\frac{\partial u_d}{\partial\nu}=0\ \ \ \ \ \ \ \ \ \ \ \ \ \ \ \ \ \ \ \ \ \ \ \ \mbox{on}\ \partial \Omega,\\
\end{cases}\end{equation*}
where $d>0$ is a parameter, $\Omega$ is a smooth bounded domain in $\mathbb{R}^2$, $\nu$ is the unit outer normal to $\partial \Omega$. We first prove the existence of a ground state solution to the above equation.   If $d$ is sufficiently small, we prove that any ground state solution $u_d$ has at most one maximum point which is located on the boundary of $\Omega$ and characterize the shape of ground state solution $u_d$ around the condensation point $P_d$. The key point of the proof lies in proving that the maximum point $P_d$ is close to the boundary at the speed of $\sqrt{d}$ when $d\rightarrow0$ and $u_d$ under suitable scaling transform converges strongly to the ground state solution of the limit equation $\Delta w+w=w(e^{w^2}-1)$.
    Our proof is based on the energy threshold of cut-off function, the concentration compactness principle for the
  Trudinger-Moser inequality, regularity theory for elliptic equation and an accurate analysis for the energy of the ground state solution $u_d$ as $d\rightarrow0$. Furthermore, by assuming that $\Omega$ is a unit disk, we remove the smallness assumption on $d$ and show the maximum point of ground state solution $u_d$ must lie on the boundary of $\Omega$ for any $d>0$.
\end{abstract}
\maketitle {\small {\bf Keywords:} Neumann elliptic problem; Trudinger-Moser inequality; Concentration phenomenon; Ground state solutions.\\

}
\section{Introduction}
The main purpose of this paper is to study the location of maximum point of a ground state solution to the semilinear Neuman elliptic problem with the critical exponential growth and characterize the shape of ground state solution $u_d$ around its condensation point $P_d$. Bubbling problem for the semilinear Neumann elliptic problem has attracted much attention due to its application to problems in biological pattern formation, such as the shadow system of some reaction-diffusion system in morphogenesis and a chemotactic aggregation model with logarithmic sensitivity (see \cite{Ke,LNT,Ta} for details). Let us first present a brief history of the main results on bubbling problems for Neumann elliptic equation when the nonlinearity is of polynomial growth.
\medskip

Let $\Omega$ be a bounded domain with smooth boundary in $\mathbb{R}^n$ ($n\geq 3$) and $\nu$ denotes the unit outer normal to $\partial \Omega$, the semilinear Neumann elliptic problem with the polynomial growth
\begin{equation}\label{equg.p}\begin{cases}
-d\Delta u_d+u_d=u^p_d\ \ \mbox{in}\ \Omega, 1<p\leq \frac{n+2}{n-2},\\
\frac{\partial u_d}{\partial\nu}=0\ \ \ \ \ \ \ \ \ \ \ \ \ \ \ \mbox{on}\ \partial \Omega,\\
\end{cases}\end{equation}
has been widely considered in the literature for more than 30 years. When $1<p<\frac{n+2}{n-2}$, the functional related to equation \eqref{equg.p} satisfies the compactness condition, it is easy to prove that there is a least energy solution for equation \eqref{equg.p} (see \cite{LNT}). Furthermore, in 1990s, Ni and Takagi (see \cite{NiW}) first studied the location of maximum point of the ground state solution $u_d$. In their work, they proved that the ground state solution $u_d$ of \eqref{equg.p} has at most one local maximum point $P_d$ in $\overline{\Omega}$ which must lie on the boundary for $d$ sufficiently small. Furthermore, $u_d$ also exhibits the concentration phenomenon at the point $P_d$. More precisely, $u_d$ around $P_d$ can be described as:
$$u_d\approx w\big(\frac{|x-P_d|}{d}\big),$$
where $w$ is the unique positive, radial solution to the problem
$$\Delta w-w+w^p=0\ in \ \mathbb{R}^n, \ \ \lim_{|x|\rightarrow+\infty}w(x)=0.$$
The asymptotic location of the point $P_d$ was also discussed in their papers and characterized as
$$\lim_{d\rightarrow0}H_{\partial\Omega}(P_d)=\max_{P\in \partial\Omega}H_{\partial\Omega}(P),$$
where $H_{\partial\Omega}$ stands for the mean curvature of $\partial\Omega$ (see \cite{NiW1}). In \cite{Wei1}, Wei studied the construction of single and multiple spike-layer patterns for this problem and proved that if there exists a $p_0\in\partial\Omega$ such that $H_{\partial\Omega}(p_0)\neq0$, then a solution of the form
$u_d\approx w\big(\frac{|x-P_d|}{d}\big)$ can be found with $P_d\rightarrow p_0$.
\medskip

For the problem \eqref{equg.p}, one can not only exhibit sequences of solutions concentrating at some point on the boundary $\partial\Omega$, but also
exhibit sequences of solutions concentrating on higher dimension sets. Indeed, Malchiodi and Montenegro \cite{M1,M2,M3} constructed a sequence of solutions producing concentration phenomenon on the $k$-dimensional submanifold $\Gamma$. More precisely, if $\Gamma=\partial\Omega$ or $\Gamma$ is an embedded closed minimal submanifold of $\partial\Omega$ and the corresponding Jacobi operator is non-singular, there exists a solution $u_d$ satisfying $u_d\approx w\big(\frac{dist(x,\Gamma)}{d}\big)$ where $w$ is the unique positive, radial solution to the problem:
$$\Delta w-w+w^p=0\ in \ \mathbb{R}^{n-k}, \ \ \lim_{|x|\rightarrow+\infty}w(x)=0.$$
\vskip0.1cm

For the critical case $p=\frac{n+2}{n-2}$, the question becomes more complicated. Adimurthi and Mancini in \cite{AM} and X. Wang in \cite{Wang} showed that for $d$ is small, the ground state solution $u_d(P_d)\rightarrow+\infty$ as $d\rightarrow0$ and $u_d$ also exhibits concentration phenomenon.  We also refer to the works of  Rey \cite{R, Rey1, Rey2}, del Pino \cite{del} and del Pino, Felmer and Wei \cite{del4}, Ni, Pan and Takagi \cite{NPT}, Gui, Wei and Winter \cite{Gui1}, Gui and Ghoussoub \cite{Gui2}, Wei \cite{Wei2, Wei3, Wei4},
Lin, Wang and Wei \cite{LWW}, Wang, Wei and Yan \cite{WWY} and references therein for singular perturbation problems with Neumann boundary conditions.

An interesting open conjecture related to the above semilinear Neumann elliptic equation states: for any $d>0$, does the non-constant ground state solution of equation \eqref{equg.p} attain its maximum at only one point $P_d\in \partial\Omega$. When $\Omega$ is a ball, Lin \cite{LCS} gave a positive answer. Applying the method of moving planes to the Neumann elliptic problem, they removed the assumption on $d$ and obtained the following result:
\vskip0.2cm


\noindent\textbf{Theorem A.} (\cite{LCS})\, \textit{Let} $u_d$
\textit{be a ground state solution of \eqref{equg.p} with } $1<p\leq\frac{n+2}{n-2}$\textit{ and }$\Omega=B_R(0)$. \textit{Suppose} $u_d$ \textit{is a nonconstant solution. Then} $u_d$ \textit{attains its local maximum at only one point} $P_d$, $P_d\in\partial B_R(0)$.  \textit{Furthermore, if we assume} $P_d=(0,\cdots,0,R)$, \textit{then} $u_d$ \textit{is increasing in} $x_n$, $u_d$ \textit{is axially symmetric with respect to} $\overrightarrow{OP_d}$, \textit{and on each sphere} $S_r=\{x:|x|=r\}$ \textit{with} $0<r\leq R$, $u_d$ \textit{is strictly decreasing as the angle of} $\overrightarrow{Ox}$ \textit{and} $\overrightarrow{OP_d}$ \textit{increases, that is}
\begin{equation*}
x_j\frac{\partial u_d}{\partial x_n}-x_n\frac{\partial u_d}{\partial x_j}>0\ \mbox{for}\ x_j>0\ \mbox{and}\ j=1,2,\cdots,n-1.
\end{equation*}
\begin{equation*}
\frac{\partial u_d}{\partial x_n}>0\ \mbox{for}\ x\in\overline{B}_R(0)\backslash \{(0,\cdots,0,\pm R\}.
\end{equation*}

It should be noted that the nonlinearity of equation \eqref{equg.p} with polynomial growth has been considered by many authors because of the Sobolev imbedding theorem: $W^{1,2}_0(\Omega)\subset L^q(\Omega)$ for $1\leq q\leq \frac{2n}{n-2}$ and $n>2$. When $n=2$, the Sobolev exponent becomes infinite and $W^{1,2}_0(\Omega)$ can be imbedded into the Orlicz space $L_{\phi_{\alpha}}$ determined by the Young function $\phi_\alpha(t)$ behaving like $e^{\alpha t^2}$ as $t\rightarrow +\infty$. A natural but non-trivial problem arises. Can Ni and Takagi's result still hold if we replace the nonlinearity of equation \eqref{equg.p} with critical exponential growth? The main purpose of this paper is to solve these problems. Because the proof of our results needs some basic theory of the Trudinger-Moser inequality, for simplicity, we will give a brief history of the Trudinger-Moser inequality.
\vskip0.2cm

The Trudinger inequality as a borderline case of the Sobolev imbedding was obtained by Trudinger \cite{Tru} (see also Pohozhaev \cite{Poh}). More precisely,
he proved that there exists $\alpha>0$ such that
\begin{equation}\label{Tru}
\sup_{\|\nabla u\|_n^n\leq 1, u\in W^{1, n}_0(\Omega)}\int_{\Omega}\exp(\alpha|u|^{\frac{n}{n-1}})dx<\infty,
\end{equation}
where $\Omega\subseteq \mathbb{R}^n$ is a smooth bounded domain and $W^{1,p}_0(\Omega)$ denotes the usual Sobolev space, i.e, the completion of $C_{0}^{\infty}(\Omega)$
with the norm $$\|u\|_{W^{1,p}_0(\Omega)}=(\int_{\Omega}\left(|\nabla u|^p+|u|^p\right)dx)^\frac1p.$$
Subsequently, Trudinger inequality was sharpened by Moser in \cite{Mo} by showing that the largest $\alpha$ in (\ref{Tru}) is  $\alpha_n=nw_{n-1}^{\frac{1}{n-1}}$, where $\omega_{n-1}$ is the surface measure of the unit sphere in $\mathbb{R}^n$.
The inequality \eqref{Tru} in the case of $\alpha=\alpha_n$ is known as the Trudinger-Moser inequality. So far, the Trudinger-Moser inequalities on bounded domains have been generalized
in other settings such as on the
CR spheres, compact Riemannian manifolds, Heisenberg
group, we refer the interested readers to
\cite{CL1}, \cite{CL2}, \cite{Fontana}, \cite{Li1}, \cite{Li2} and the
references therein.
\vskip0.1cm

The concentration phenomenon and the singular perturbation problems with Dirichlet boundary conditions when the  nonlinearities are of  exponential growth have also been studied by  Struwe \cite{Struwe}, Adimurthi and Struwe \cite{AdimurthiStruwe}, Lamm, Robert and Struwe \cite{LRS},   del Pino, Musso and Ruf \cite{del2, del3}, Druet \cite{Druet1}, Druet and Thizy \cite{DT},  Malchiodi and Martinazzi \cite{MM}, Marchis, Malchiodi, Martinazzi and Thizy \cite{MMMT} and references therein.
\medskip

In this paper, we are interested in investigating the concentration phenomenon and the singular perturbation problems with Neumann boundary conditions when the  nonlinearities are of  critical exponential growth.

\medskip
Now we are in a position to explain our main results. We focus on the positive solution of the following semilinear Neumann elliptic equation with the Trudinger-Moser growth:
\begin{equation}\label{equg.1}\begin{cases}
-d\Delta u+u=u(e^{u^2}-1)\ \ \mbox{in}\ \Omega,\\
\frac{\partial u}{\partial\nu}=0\ \ \ \ \ \ \ \ \ \ \ \ \ \ \ \ \ \ \ \ \ \ \ \ \mbox{on}\ \partial \Omega,\\
\end{cases}\end{equation}
where $d>0$, $\Omega$ is a smooth bounded domain in $\mathbb{R}^2$, $\nu$ is the unit outer normal to $\partial \Omega$. Obviously, the functional $J_d:W^{1,2}(\Omega)\rightarrow\mathbb{R}$ associated with equation \eqref{equg.1} is defined by
\begin{equation}\label{equg.2}
J_d(v)=\frac12\int_\Omega (d|\nabla v|^2+v^2)dx-\frac12\int_\Omega(\exp(v^2)-v^2-1)dx
\end{equation}
and $J_d'(u)v=d\int_{\Omega}\nabla u\nabla vdx+\int_{\Omega}uvdx-\int_{\Omega}(e^{u^2}-1)uvdx$ with $u, v\in W^{1,2}(\Omega)$.
We recall that the solution $u_d$ of equation \eqref{equg.2} is called a ground state solution if $J_d(u_d)=m_d:=\inf\{J_{d}(u): J_d'(u)=0\}$. It is also easy to check that the functional energy of the ground state solution $u_d$ can also be characterized by min-max technique, that is \begin{equation}\label{equg.3}
m_d:=\inf_{h\in \Gamma}\max_{0\leq t\leq1} J_d(h(t)),\ \ \Gamma:=\{h\in \mathcal{C}([0,1],W^{1,2}(\Omega)):h(0)=0,J_d(h(1))<0\}.
\end{equation}

To our knowledge, we have not seen any existence result for ground state solutions of equation \eqref{equg.1}. For the reader's convenience, we first establish the existence of ground state solutions to this equation.

\begin{theorem}\label{thm1}
For any $d>0$, the equation \eqref{equg.2} admits a positive ground state solution $u_d$.
\end{theorem}

\begin{remark}
 The proof combines the Trudinger-Moser inequality, the concentration-compactness principle for Trudinger-Moser inequality in $W^{1,2}(\Omega)$ and Nehari manifold method.
\end{remark}

\vskip 0.1cm

Obviously, equation \eqref{equg.1} has two constant solutions $u\equiv0$ and $u\equiv\ln^\frac12 2$. Remark \ref{nonconstant} yields that $u_d\not\equiv C$ provided that $d$ is sufficiently small. Furthermore, we also obtain
\begin{theorem}\label{theorem2.2}
Let $u_d$ be a ground state solution of \eqref{equg.1}. If $d$ is sufficiently small, then $u_d$ has at most one local maximum at $P_d\in\overline{\Omega}$. Furthermore, $P_d$ must lie on the boundary $\partial\Omega$.
\end{theorem}
The proof of Theorem \ref{theorem2.2} is quite involved. The key idea is to prove that the distance between the maximum point $P_d$ and the boundary satisfies $d(P_d,\partial\Omega)\lesssim d^\frac12$ and show that
$u_d$ under suitable scaling transform converges strongly to the ground state solution of the limit equation $-\Delta w+w=w(e^{w^2}-1)$. The presence of critical exponential growth makes this problem nontrivial. One can not just follow the same line of Ni and Tagaki \cite{NiW} to obtain the desired conclusion. Our proof combines the gradient estimate of cut-off function, the concentration compactness principle for Trudinger-Moser inequality, regularity theory for elliptic equation and a more accurate analysis for the energy of the ground state solution $u_d$ as $d\rightarrow0$.
\vskip 0.1cm

\begin{theorem}\label{theorem2.3}
Suppose that $u_d$ is a ground state solution of \eqref{equg.1} which achieves its maximum at $P_d\in\partial\Omega$. Then for any $\varepsilon>0$, there is a constant $d_0$ and a subdomain $\Omega_d^{(\varepsilon)}\subset\Omega$ such that for $0<d<d_0$, there holds:
\vskip 0.1cm

(i) $P_d\in\partial\Omega_d^{(\varepsilon)}$ and diam$(\Omega_d^{(\varepsilon)})\leq C\sqrt d$,
\medskip

(ii) $\|u_d(\cdot)-w(\Psi(\cdot)/\sqrt d)\|_{C^2(\overline{\Omega_d^{(\varepsilon)}})}\leq\varepsilon$,
\medskip

(iii) \begin{equation*}|u_d(x)|\leq C_1 \varepsilon \exp(-\mu_1\delta(x)/\sqrt d)\mbox{ for }x\in \Omega_d^{(c)}:=\Omega\backslash\Omega_d^{(\varepsilon)},
\end{equation*}
 where $\Psi(x)$ is a function defined in \eqref{equg.11}, $\delta(x)=\min \{d (x,\partial\Omega_d^{(\varepsilon)}),\eta_0\}$ and $C$, $C_1$, $\mu_1$ and $\eta_0$ are positive constants depending only on $\Omega$.
\end{theorem}

Theorem \ref{theorem2.3} characterizes the concentration behavior of ground state solution $u_d$ around the maximum point $P_d$ when $d$ is small. When $\Omega$ is a special region such as unit disk, we can remove the smallness assumption on $d$ and prove that for any $d>0$, the maximum-point of ground state solution $u_d$ must lie on the boundary $\partial D$. Indeed, applying the local moving-plane method developed by Lin in \cite{LCS}, we obtain the following result:

\begin{theorem}\label{thm.1}
Suppose $u_d$ is a ground state solution of equation \eqref{equg.1} with $\Omega$ replaced by unit disk $D$ in $\mathbb{R}^2$. If $u_d$ is a nonconstant solution, then for any $d>0$, $u_d(x)$ has only one extremal point $P_d$ which achieves its maximum. Furthermore, $P_d$ must lie on the boundary $\partial D$.
\end{theorem}

This paper is organized as follows. In Section 2, we apply the concentration compactness principle for Trudinger-Moser inequality in $W^{1,2}(\Omega)$ and Nehari manifold method to get the existence of ground state solutions and give the proof of Theorem \ref{thm1}.
Section 3 is devoted to some necessary lemmas. In the spirit of Berestycki and Lions' work \cite{BL} and combining the method of moving planes in integral form developed by Chen, Li and Ou \cite{CLO}, we obtain that any positive solutions of Schr\"{o}dinger equation with the Trudinger-Moser growth \eqref{equg.2} is radial and decays exponentially at infinity.
By constructing an appropriate sequence $\varphi_d$ and computing $M[\varphi_d]$, we establish the relationship between $m_d$, $I(w)$ and the mean curvature of $\partial \Omega$ at $P$, where $w=w(z)$ is the ground state solution of Schr\"{o}dinger equation \eqref{equg.2} and $P$ is the limiting point of $P_d$. 
 In Section 4, we establish the phenomenon of point condensation and show the shape of the ground state solution $u_d$ around the condensation point $P_d$.
 In Section 5, we consider phenomenon of point condensation on unit ball $D$ and
show that for every $d>0$, $u_d$ has only one extremal point $P_d\in \partial D$ through the local moving-plane method in Theorem \ref{thm.1}.
\section{The Proof of Theorem \ref{thm1}}
In this section, we will apply Nehari manifold method and concentration compactness principle for Trudinger-Moser inequality in $W^{1,2}(\Omega)$ to prove that equation \eqref{equg.1} has a positive ground state solution. For this purpose, we introduce
the functional
$$G_d(u)=J'_{d}(u)u=\int_{\Omega}(d|\nabla u|^2+|u|^2)dx-\int_{\Omega}u^2 (e^{u^2}-1)dx$$
and the constrained minimization problem
\begin{equation}
m_{d} :=\inf\Big\{\frac{1}{2}\int_{\Omega}u^2 (e^{u^2}-1)dx-\frac{1}{2}\int_{\Omega}\big(e^{u^2}-1-u^2\big)dx \ \big| \ u\in W^{1,2}(\Omega),\ G_{d}(u)=0 \Big\}.
\end{equation}
If $m_{d}$ could be achieved by a function $u_d$ in $W^{1,2}(\Omega)$, then $u_d$ is a ground state solution of \eqref{equg.1}.
\vskip0.2cm

Set $M=\Big\{u\in W^{1,2}(\Omega)\ |\ \ G_{d}(u)=0\Big\}$, we point out that $M$ is not empty. In fact, let
$u_0\in W^{1,2}(\Omega)$ be compactly supported and for any $s>0$, define
$$h(s):=G_{d}(su_0)=s^2\int_{\Omega}(d|\nabla u_0|^2+|u_0|^2)dx-\int_{\Omega}(su_0)^2(e^{(su_0)^2}-1)dx.$$
Obviously $h(s)>0$ for $s>0$ small enough and $h(s)<0$ for $s>0$ sufficiently large. Therefore, there exists $s_0>0$ satisfying
$h(s_0u_0)=0$, which implies $s_0u_0 \in M$.
\begin{lemma}\label{minimaxleval}
$$0<m_{d}<\pi d.$$
\end{lemma}
\begin{proof}
We first show that $m_{d}>0$. Assume that $m_{d}=0$, then there exists a sequence $\{u_k\}_k\subseteq W^{1,2}(\Omega)$ such that
\begin{equation*}\int_{\Omega}(d|\nabla u_k|^2+|u_k|^2)dx-\int_{\Omega}u_k^2\big(\exp(u_k^2)-1\big)=0,\ \forall k\in\mathbb{N}^*\end{equation*}
and
\begin{equation*}\lim\limits_{k\rightarrow \infty}\frac{1}{2}\int_{\Omega}(d|\nabla u_k|^2+|u_k|^2)-\frac{1}{2}\int_{\Omega}\big(\exp(u_k^2)-1-u_k^2\big)dx=0.\end{equation*}
Direct computations give
\begin{equation}\begin{split}
m_{d}&=\lim\limits_{k\rightarrow \infty}\big(\frac{1}{2}\int_{\Omega}(d|\nabla u_k|^2+|u_k|^2)dx-\frac{1}{2}\int_{\Omega}\big(\exp(u_k^2)-1-u_k^2\big)dx\big)\\
&=\lim\limits_{k\rightarrow \infty}\big(\frac{1}{2}\int_{\Omega}u_k^2\big(\exp(u_k^2)-1\big)dx-\frac{1}{2}\int_{\Omega}(\exp(u_k^2)-1-u_k^2)dx\big)\\
&\geq \frac{1}{4}\lim\limits_{k\rightarrow \infty}\int_{\Omega}u_k^2 (\exp(u_k^2)-1)dx\\
&=\frac{1}{4}\lim\limits_{k\rightarrow \infty}\int_{\Omega}(d|\nabla u_k|^2+|u_k|^2)dx.
\end{split}\end{equation}
Therefore, it follows from the Sobolev imbedding theorem that
$$u_k\rightharpoonup 0\ {\rm  in}\ W^{1,2}(\Omega)\ \ {\rm and}\ \ u_k\rightarrow 0\ \ {\rm in}\ L^{p}(\Omega)\ \ {\rm for\ any }\ p\geq 1.$$
Let $v_k=\frac{u_k}{(d\|\nabla u_k\|_2^2+\|u_k\|_2^2)^{\frac{1}{2}}}$ which weakly converges to $v$, we derive
\begin{equation}\begin{split}
1&=\int_{\Omega}\frac{u_k^2}{(d\|\nabla u_k\|_2^2+\|u_k\|_2^2)}(\exp(u_k^2)-1)dx\\
&=\int_{\Omega}(\exp(u_k^2)-1)|v_k|^2dx\rightarrow 0
\end{split}\end{equation}
which is a contradiction and $m_d>0$. Next we start to prove that $m_{d}<\pi d$. Let $w\in W^{1,2}(\Omega)$ such that $d\| \nabla w\|_2^2+\|w\|_2^2=1$.
Then there exists $\gamma_{w}>0$ such that $$\int_{\Omega}(d|\nabla \gamma_{w} w|^2+|\gamma_{w} w|^2)dx-\int_{\Omega}(\gamma_{w} w)^2\big(e^{(\gamma_{w} w)^2}-1\big)dx=0,$$
which implies that
\begin{equation}\begin{split}
m_{d}&\leq \frac{1}{2}\int_{\Omega}\big(d|\nabla_{\Omega}\gamma_{w} w|^2+|\gamma_{w} w|^2\big)dx
-\frac{1}{2}\int_{\Omega}\big(e^{(\gamma_{w} w)^2}-1-(\gamma_{w} w)^2\big)dx\\
&<\frac{\gamma_{w}^2}{2}\int_{\Omega}\big(d|\nabla w|^2+|w|^2\big)dx=\frac{\gamma_{w}^2}{2}.
\end{split}\end{equation}
On the other hand, $\big(e^{(\gamma w)^2}-1\big)w^2$ is monotone increasing about the variable $\gamma$. Set $m_{d}=\frac{\gamma_{\infty}^2}{2}$, then we derive that
\begin{equation}\begin{split}
\int_{\Omega}\big(e^{(\gamma_{\infty} w)^2}-1\big)w^2dx&\leq \int_{\Omega}\big(e^{(\gamma_{w} w)^2}-1\big)w^2dx\\
&=\int_{\Omega}(d|\nabla  w|^2+|w|^2)dx=1,
\end{split}\end{equation}
which implies that $$\sup_{\int_{\Omega}(d|\nabla  w|^2+|w|^2)dx=1}\int_{\Omega}\big(e^{(\gamma_{\infty} w)^2}-1\big)w^2dx<\infty.$$
This together with the Trudinger-Moser inequality in $W^{1,2}(\Omega)$ (see Lemma \ref{sharp}) leads to $m_{d}=\frac{\gamma_{\infty}^2}{2}<\pi d$.
\end{proof}
\begin{remark}\label{nonconstant}
Based on Lemma \ref{minimaxleval}, one can get $u_d$ is a nonconstant solution. Indeed, suppose $u_d$ is a constant solution of equation \eqref{equg.1}, then $u_d\equiv0$ or $u_d\equiv\ln^\frac122$. Once $u_d\equiv0$, direct calculations show that $$m_d=\frac{1}{2}\int_{\Omega}(d|\nabla u_d|^2+|u_d|^2)dx-\frac{1}{2}\int_{\Omega}\big(\exp(u_d^2)-1-u_d^2\big)dx=0,$$ which contradicts with $m_d>0$. Moreover, $u_d\equiv\ln^\frac122$ can not hold either. Suppose $u_d\equiv\ln^\frac122$, then
$$m_d=\frac{1}{2}\int_{\Omega}(d|\nabla u_d|^2+|u_d|^2)dx-\frac{1}{2}\int_{\Omega}\big(\exp(u_d^2)-1-u_d^2\big)dx=(\ln2-\frac12)|\Omega|,$$
 which is a contradiction with $m_d<\pi d$ provided $d$ sufficiently small. Therefore $u_d$ is a nonconstant solution. Furthermore, since $u_d$ is a ground state solution, we have $u_d\geq0$.
\end{remark}
\begin{lemma}\label{sharp}
Let $\Omega$ be a smooth bounded domain and define $\mathcal{H}_2$ by
$$\mathcal{H}_2=\Big\{u\in W^{1,2}(\Omega)\Big|\int_\Omega(|\nabla u|^2+|u|^2)dx=1\Big\}.$$
Then there holds
\begin{equation}\label{sharp1}
\sup_{u\in\mathcal{H}_2}\int_\Omega u^2(\exp(\alpha u^2)-1)dx<+\infty
\end{equation}
if $\alpha<2\pi$ and
\begin{equation}\label{sharp2}
\sup_{u\in\mathcal{H}_2}\int_\Omega u^2(\exp(2\pi u^2)-1)dx=+\infty
\end{equation}
if $\alpha=2\pi$.
\end{lemma}
\begin{proof}
When $\alpha<2\pi$, one can apply the H\"{o}lder inequality and Lemma 6.1 in \cite{YYY} to obtain \eqref{sharp1}. Then it suffices to show that $2\pi$ is the sharp constant. Fix $p\in \partial \Omega$ and define 
\begin{equation}
u_\varepsilon(x)=\begin{cases}
(\frac{1}{2\pi})^\frac12(\log\frac1\varepsilon)^\frac12\ \ \ \ \ \ \ \ \ \ \ \ \ \ \ \mbox{in}\ B_{\delta\sqrt\varepsilon}(p),\\
(\frac{2}{-\pi\log\varepsilon})^\frac12\log(\frac{\delta}{|x-p|})\ \ \ \ \ \ \ \mbox{in}\
B_{\delta}(p)\backslash B_{\delta\sqrt\varepsilon}(p),\\
0\ \ \ \ \ \ \ \ \ \ \ \ \ \ \ \ \ \ \ \ \ \ \ \ \ \ \ \ \ \ \mbox{in}\ \Omega\backslash B_{\delta}(p).
\end{cases}\end{equation}
Direct calculations show that
$$\int_{\Omega}|\nabla  u_\varepsilon|^2dx=1,\ \ \ \int_{\Omega}|u_\varepsilon|^2dx=O(\frac1{\log\frac1\varepsilon})$$
and
\begin{equation*}\begin{split}
&\int_{\Omega} (\frac{u_\varepsilon}{\|u_\varepsilon\|_{W^{1,2}(\Omega)}})^2(\exp(2\pi (\frac{u_\varepsilon}{\|u_\varepsilon\|_{W^{1,2}(\Omega)}})^2)-1)dx\\
&\geq
\int_{B_{\delta\sqrt\varepsilon}(p)} (\frac{u_\varepsilon}{\|u_\varepsilon\|_{W^{1,2}(\Omega)}})^2(\exp(2\pi (\frac{u_\varepsilon}{\|u_\varepsilon\|_{W^{1,2}(\Omega)}})^2)-1)dx\\
&\gtrsim\frac{\log\frac{1}{\varepsilon}}{1+\frac1{\log\frac1\varepsilon}}
(\exp(\frac{\log\frac{1}{\varepsilon}}{1+\frac1{\log\frac1\varepsilon}})-1)|B_{\delta\sqrt\varepsilon}(p)|\\
&\gtrsim\log\frac{1}{\varepsilon}.
\end{split}\end{equation*}
Therefore, one can obtain that $\sup_{u\in\mathcal{H}_2}\int_\Omega u^2(\exp(2\pi u^2)-1)dx=+\infty$ which completes the proof.
\end{proof}
\begin{lemma}\label{lem3}
Let $\{u_k\}_k$ be a bounded sequence in $W^{1,2}(\Omega)$ which converges weakly to $u_d$ such that
$$\sup_{k}\int_{\Omega} u_k^2\big(e^{u_k^2}-1\big)dx<\infty,$$ then
$$\lim_{k\rightarrow \infty}\int_{\Omega}\big(e^{u_k^2}-1-u_k^2\big)dx=\int_{\Omega}\big(e^{u^2_d}-1-u^2_d\big)dx.$$
\end{lemma}
\begin{proof}
Up to a sequence, $\{u_k\}_k$ converges to $u_d$ for almost every $x\in \Omega$. Dividing the integral into two parts, we have
\begin{equation}\begin{split}
&\int_{\Omega}\big(e^{u_k^2}-1-u_k^2\big)dx-\int_{\Omega}\big(e^{u^2_d}-1-u^2_d\big)dx\\
&\ \ =\Big(\int_{\{|u_k|\leq R\}} \big(e^{u_k^2}-1-u_k^2\big)dx-\int_{\{|u|\leq R\}}\big(e^{u^2_d}-1-u^2_d\big)dx\Big)\\
&\ \ \ \ \ +\Big(\int_{\{|u_k|\geq R\}}\big(e^{u_k^2}-1-u_k^2\big)dx-\int_{\{|u|\geq R\}}\big(e^{u^2_d}-1-u^2_d\big)dx\Big)\\
&\ \ =:I_{k,R}+II_{k,R}.\\
\end{split}\end{equation}
For $I_{k,R}$, dominated convergence theorem yields that $\lim\limits_{k\rightarrow \infty}I_{k,R}=0$.
As for $II_{k,R}$, since
\begin{equation}\begin{split}
\int_{\{|u_k|\geq R\}}\big(e^{u_k^2}-1-u_k^2\big)dx&\leq \frac{1}{R^2}\int_{\Omega}u_k^2\big(e^{u_k^2}-1\big)dx\\
&\leq \frac{1}{R^2}\sup_{k}\int_{\Omega}u_k^2\big(e^{u_k^2}-1\big)dx,
\end{split}\end{equation}
then $\lim\limits_{R\rightarrow \infty}\lim\limits_{k\rightarrow\infty}II_{k,R}=0$.
Combining the above estimates, we accomplish the proof of Lemma \ref{lem3}.
\end{proof}
\begin{lemma}\label{compact}[Compactness Lemma] Let $\{u_k\}_k$ be a sequence satisfying $G_d(u_k)=0$ and $J_d(u_k)\rightarrow m_d$. Assume that $\{u_k\}_k$ is a bounded sequence in $W^{1,2}(\Omega)$ which converges weakly to a non-zero function $u_d$ and $\int_{\Omega}(d|\nabla u_d|^2+|u_d|^2)dx>\int_{\Omega}u^2_d\big(e^{u^2_d}-1\big)dx$, then
$$\lim_{k\rightarrow \infty}\int_{\Omega}u_k^2\big(e^{u_k^2}-1\big)dx=\int_{\Omega}u^2_d\big(e^{u^2_d}-1\big)dx.$$
\end{lemma}
\begin{proof}
Up to a sequence, $\{u_k\}_k$ converges to $u_d$ for almost every $x\in \Omega$.
By the lower semicontinuity of the norm in $W^{1,2}(\Omega)$, we have
$$\lim\limits_{k\rightarrow\infty}\int_{\Omega}(d|\nabla u_k|^2+|u_k|^2)dx\geq \int_{\Omega}(d|\nabla u_d|^2+|u_d|^2)dx.$$

Case 1: $\lim\limits_{k\rightarrow\infty}\int_{\Omega}(d|\nabla u_k|^2+|u_k|^2)dx=\int_{\Omega}(d|\nabla u_d|^2+|u_d|^2)dx$. According to the convexity of the norm in $W^{1,2}(\Omega)$, we see that
$u_k\rightarrow u_d$ in $W^{1,2}(\Omega)$, hence $u_k\rightarrow u_d$ in $L^{p}(\Omega)$ for any $p\geq 1$. Then
it follows from the Trudinger-Moser inequality in $W^{1,2}(\Omega)$ that for any $p_0>1$, $\sup_{k}\int_{\Omega} \big(u_k^2 \exp(|u_k|^2)\big)^{p_0}dx<\infty$,
which together with Vitali convergence Theorem yields
\begin{equation}\label{con1}
\lim_{k\rightarrow\infty}\int_{\Omega}u_k^2\big(e^{u_k^2}-1\big)dx=\int_{\Omega}u^2_d\big(e^{u^2_d}-1\big)dx.
\end{equation}
\vskip0.2cm

Case 2: If $\lim\limits_{k\rightarrow\infty}\int_{\Omega}(d|\nabla u_k|^2+|u_k|^2)dx>\int_{\Omega}(d|\nabla u_d|^2+|u_d|^2)dx$, we set
 $$ v_k:=\frac{u_k}{\lim\limits_{k\rightarrow\infty}\int_{\Omega}(d|\nabla u_k|^2+|u_k|^2)dx}\ \mbox{and}\ v_0:=\frac{u_d}{\lim\limits_{k\rightarrow\infty}\int_{\Omega}(d|\nabla u_k|^2+|u_k|^2)dx}.$$
 We claim that there exists $q_0>1$ sufficiently close to $1$ such that
 \begin{equation}\label{d.7}
  q_0\lim_{k\rightarrow\infty}(\int_{\Omega}(d|\nabla u_k|^2+|u_k|^2)dx)<\frac{2\pi d}{1-\int_{\Omega}(d|\nabla v_0|^2+|v_0|^2)dx}.
  \end{equation}
 On the other hand, we can also apply $\int_{\Omega}(d|\nabla u_d|^2+|u_d|^2)dx>\int_{\Omega}u^2_d\big(e^{u^2_d}-1\big)dx$ to obtain $J_d(u_d)>0$, which together with $J_d(u_k)\rightarrow m_d$ and Lemma \ref{minimaxleval} yields that
  \begin{equation}\begin{split}\label{d.8}
  &\lim\limits_{k\rightarrow\infty}\int_{\Omega}(d|\nabla u_k|^2+|u_k|^2)dx\big(1-(\int_{\Omega}(d|\nabla v_0|^2+|v_0|^2)dx)\big)\\
 &\ \ =2m_d+\lim\limits_{k\rightarrow\infty}\int_{\Omega}(e^{u_k^2}-1-u_k^2)dx-2J_d(u_d)-\int_{\Omega}(e^{u^2_d}-1-u^2_d)dx\\
&\ \ <2\pi d.
\end{split}\end{equation}
Combining \eqref{d.7} with the concentration compactness principle for Trudinger-Moser inequality in $W^{1,2}(\Omega)$, one can derive that there exists $p_0>1$ such that
\begin{equation}\label{d.9}
\sup_{k}\int_{\Omega}\big(e^{u_k^2}-1\big)^{p_0}dx<\infty.
\end{equation}
Then it follows from the similar progress as Lemma \ref{lem3} that
$$\lim_{k\rightarrow\infty}\int_{\Omega}u_k^2\big(e^{u_k^2}-1\big)dx=\int_{\Omega}u^2_d(e^{u^2_d}-1)dx.$$
\end{proof}
\begin{remark} The concentration-compactness principle for functions in $W^{1,2}_0(\Omega)$ was established in \cite{Ce} and \cite{Lo}.
The Lions type  concentration-compactness principle  for Trudinger-Moser inequality for functions  in $W^{1,2}(\Omega)$ without compact support in $\Omega$ states if $\{u_k\}_k\subseteq \mathcal{H}_2$ satisfying $u_k\rightharpoonup u\neq 0$ in $W^{1,2}(\Omega)$, then for any $p<\frac{1}{1-\int_{\Omega}(|\nabla u|^2+|u|^2)dx}$, there holds
$$\int_{\Omega}e^{2\pi pu_k^2}dx<+\infty.$$ Since $W^{1,2}(\Omega)$ has the Hilbert space structure, the proof is easily verified by combining the subcritical Trudinger-Moser inequality and property of weak convergence in $W^{1,2}(\Omega)$ by applying a similar argument to that in  \cite{LiLu, LLZcv, ZC} where a symmetrization-free argument initially developed in \cite{LL7, LL2} was used.

\end{remark}
\textbf{Existence of ground state Solutions}
Now we are in a position to prove that $m_d$ could be achieved by a non-zero function $u_d\in W^{1,2}(\Omega)$.
We first claim that $u_d\neq 0$ and argue this by contradiction. Suppose that $u_d=0$, then
\begin{equation}\begin{split}
\lim\limits_{k\rightarrow\infty}\int_{\Omega}(d|\nabla u_k|^2+|u_k|^2)dx&=2\lim_{k\rightarrow \infty}J_{d}(u_k)+\int_{\Omega}\big(e^{u_d^2}-1-u_d^2\big)dx\\
&=2\lim_{k\rightarrow \infty}J_{d}(u_k)<2\pi d.
\end{split}\end{equation}
This together with Trudinger-Moser inequality and Vitali convergence theorem yields that
$$\lim_{k\rightarrow \infty}\int_{\Omega} u_k^2 \big(e^{u_k^2}-1\big)dx=0\ \mbox{and}\ \lim_{k\rightarrow \infty}\int_{\Omega} \big(e^{u_k^2}-1-u_k^2 \big)dx=0,$$
which implies that $$0<2m_{d}=\lim_{k\rightarrow \infty}\int_{\Omega}(d|\nabla u_k|^2+|u_k|^2)dx=\lim_{k\rightarrow \infty}\int_{\Omega}u_k^2 \big(e^{u_k^2}-1\big)dx=0.$$
Thus we get a contradiction. This proves $u_d\neq 0$.

Next, we claim that
\begin{equation}\label{t1}\int_{\Omega}(d|\nabla u_d|^2+|u_d|^2)dx\leq \int_{\Omega}u_d^2\big(e^{u_d^2}-1\big)dx.\end{equation}
Suppose not, in view of Lemma \ref{compact}, we derive that  \begin{equation}\label{t3}\lim_{k\rightarrow \infty}\int_{\Omega}u_k^2 \big(e^{u_k^2}-1\big)dx=\int_{\Omega} u_d^2\big(e^{u^2_d}-1\big)dx.\end{equation}
Hence $$\int_{\Omega}(d|\nabla u_d|^2+|u_d|^2)dx>\int_{\Omega}u^2_d\big(e^{u^2_d}-1\big)dx=\lim_{k\rightarrow \infty}\int_{\Omega}u_k^2 \big(e^{u_k^2}-1\big)dx=\lim_{k\rightarrow \infty}\int_{\Omega}(d|\nabla u_k|^2+|u_d|^2)dx,$$
which is a contradiction. Hence, there exists $\gamma_{d}\in (0,1]$ such that $\gamma_{d}u_d\in M$. According to the definition of  $m_{d}$, we derive that
\begin{equation}\begin{split}
m_{d}&\leq J_{d}(\gamma_{d}u_d)=\frac{1}{2}\int_{\Omega}(\gamma_{d}u_d)^2 \big(e^{(\gamma_du_d)^2}-1\big)dx-\frac{1}{2}\int_{\Omega}
\big(e^{(\gamma_du_d)^2}-1-(\gamma_du_d)^2\big)dx\\
&\leq \frac{1}{2}\int_{\Omega}u_d^2 \big(e^{u^2_d}-1\big)dx-\frac{1}{2}\int_{\Omega}
\big(e^{u^2_d}-1-u^2_d\big)dx\\
&\leq \lim_{k\rightarrow \infty}\frac{1}{2}\int_{\Omega}u_k^2 \big(e^{u_k^2}-1\big)dx-\frac{1}{2}\int_{\Omega}
\big(e^{u_k^2}-1-u_k^2\big)dx\\
&=\lim_{k\rightarrow \infty}J_{d}(u_k)=m_{d}.
\end{split}\end{equation}
This implies that $\gamma_d=1$, $u_d\in M$ and $I_{\lambda}(u_d)=m_{d}$. Thus Theorem \ref{thm1} is proved.
\section{Some necessary lemmas}
In this section, we give some lemmas which play a key role in the proofs of Theorem \ref{theorem2.2} and \ref{theorem2.3}.
First, we claim that the positive solution of equation
\begin{equation}\label{equg.b2}\begin{cases}
-\Delta u+u=u(e^{u^2}-1)\ \ \mbox{in}\ \mathbb{R}^2,\\
u\in W^{1,2}(\mathbb{R}^2)
\end{cases}\end{equation}
is radially symmetric up to some translation and decays exponentially at infinity.
\begin{lemma}\label{estimate}
Assume $u$ is a weak solution of \eqref{equg.2}, then $u$ is a radial solution satisfying $\lim\limits_{|x|\rightarrow+\infty}u(x)=0$. Moreover, $|u(r)|$ and $|u'(r)|$ decay exponentially at infinity. i.e. there exists $\theta>0$ such that
$|u|,|u'(r)|\leq e^{-\theta r}$ for $r$ sufficiently large.
\end{lemma}
\begin{proof}
Let $u\in W^{1,2}(\mathbb{R}^2)$ be a solution of $-\Delta u+u=u(e^{u^2}-1)$, by Green's representation formula, we can write
$$u(x)=\int_{\mathbb{R}^2}G(x-y)u(e^{u^2}-1)dy,$$
where $G(x-y)$ is the Green function of the Schr\"{o}dinger operator $-\Delta+I$ in $\mathbb{R}^2$ and decays exponentially at infinity. Using Trudinger-Moser inequality in $W^{1,2}(\mathbb{R}^2)$ and Lebesgue dominated convergence theorem, one can easily obtain
$\lim\limits_{|x|\rightarrow+\infty}u(x)=0$. Furthermore, one can use the method of moving planes in the integral form as developed by Chen, Li and Ou \cite{CLO} as done in \cite{BLL} to
derive that $u$ is radial. Now, we adopt the method originally appeared in the work of Berestycki and Lions' paper \cite{BL} to prove that
$u$ and $|u'(r)|$ decay exponentially at infinity.
\medskip

Denote $v(x)=|x|^\frac12u(x)$ and $g(x)=xe^{x^2}-2x$. Direct calculations give that
$$v_r=\frac12r^{-\frac12}u+r^\frac12u_r$$
and
\begin{equation}\label{es.1}
v_{rr}=-\frac14r^{-\frac32}u+r^{-\frac12}u_r+r^\frac12u_{rr}
=(-\frac{g(u)}{u}-\frac14r^{-2})v.
\end{equation}
Combining $\lim\limits_{r\rightarrow+\infty}u(r)=0$, $\lim\limits_{s\rightarrow0}\frac{g(s)}{s}=-1$ with equation \eqref{es.1}, one can see that for $r$ sufficiently large, there holds
\begin{equation}\label{es.2}
v_{rr}\geq\frac12v.
\end{equation}
Define $w=v^2$, simple calculations show that
\begin{equation}\begin{split}\label{es.3}w_r=2vv_r,\ \
\frac12w_{rr}=v_r^2+vv_{rr}.
\end{split}\end{equation}
Thanks to equation \eqref{es.2}, we deduce that for $r>r_0$,
\begin{equation}\begin{split}\label{es.4}
\frac12w_{rr}\geq v_r^2+\frac12v^2
\geq\frac12 w,
\end{split}\end{equation}
where $r_0$ is a positive constant such that $|\frac{g(u(r))}{u(r)}+1|\leq \frac14$ for $r>r_0$.
Therefore, we have $w\geq0$ and $w_{rr}>w$ for $r>r_0$. Let $z=e^{-r}(w_r+w)$. Then $z_r=e^{-r}(w_{rr}-w)$. It follows from \eqref{es.4} that
$$z_r>0\ \mbox{ for}\ r>r_0.$$
 We claim that for any $r>r_0$, $z(r)\leq0$. Assume there exists $r_1>r_0$ such that $z(r_1)>0$, then $z(r)\geq z(r_1)>0$ for $r>r_1$ which yields that
\begin{equation}\label{es.5}\begin{split}
z(r_1)e^r\leq w_r+w
=2vv_r+v^2
=u^2+ruu_r+ru^2.
\end{split}\end{equation}
Since $\lim\limits_{r\rightarrow+\infty}u(r)=0$, then
$\frac{z(r_1)e^r}{r^3}$ is not integrable on $\mathbb{R}^2$ and $\frac{u^2+ruu_r+ru^2}{r^3}$ is integrable on $\mathbb{R}^2$. This is a contradiction. Hence
$$z(r)\leq0\ \mbox{for}\ r\geq r_0,$$
 which implies that
\begin{equation}\label{es.6}
(e^rw)_r=e^{2r}z\leq 0\ \mbox{for\ } r\geq r_0.
\end{equation}
As a result, we can get that $w(r)\leq ce^{-r}$ for $r\geq r_0$. Thus
\begin{equation}\label{es.7}
|u(r)|\leq\frac{w^\frac12}{r^\frac12}\lesssim r^{-\frac12}e^{-\frac{r}{2}}.
\end{equation}
As for $u_r$, we focus on $ru_r$ and obtain that
\begin{equation}\label{es.8}\begin{split}
(ru_r)_r=u_r+ru_{rr}=-rg(u(r)).
\end{split}\end{equation}
Since $|\frac{g(u(r))}{u(r)}+1|\leq \frac14$ for $r>r_0$, one can get that $\frac34|u|\leq |g(u(r))|\leq\frac54|u|$. Therefore we have for $R>r>r_0$, there holds
\begin{equation}\label{es.9}\begin{split}
Ru_R-ru_r=\int_r^R-sg(u(s))ds\thickapprox\int_r^R-su(s)ds.\\
\end{split}\end{equation}
With the help of equation \eqref{es.7}, we see that the function $ru_r$  is convergent as $r\rightarrow+\infty$ and $\lim\limits_{r\rightarrow+\infty}ru_r=0$. Then one can calculate that for $r\geq r_1$ and $\theta<\frac12$,
\begin{equation}\label{es.10}
|ru_r|=|\int_r^{+\infty}-sg(u(s))ds|
\thickapprox |\int_r^{+\infty}su(s)ds|
\leq\int_r^{+\infty}e^\frac{-r}{2}ds
\leq e^{-\theta r}.
\end{equation}
This completes the proof.
\end{proof}
\vskip0.1cm

 Besides estimates of solution $u$ and its derivative $u_r$, we also need to focus on $m_d$.
\begin{lemma}\label{lem3.1}(\cite{MW})
Let $m_d$ defined as \eqref{equg.3} is a ground state point of $J_d$, then $m_d$ could be also characterized by
\begin{equation}\label{equg.4}
m_d=\inf\{M[v]\big|v\in W^{1,2}(\Omega), v\not\equiv0 \mbox{\ and\ } v\geq0\mbox{\ in\ } \Omega\},
\end{equation}
where $M[v]:=\sup\limits_{t\geq0}J_d(tv)$.
\end{lemma}

Then, we are going to obtain the estimate $m_d\leq d\Big(\frac12I(w)-\varphi''(0)\gamma d^\frac12+o(d^\frac12)\Big)$ where $
I(w)=\frac12\int_{\mathbb{R}^2}(|\nabla w|^2+w^2)dx-\frac12\int_{\mathbb{R}^2}\big(\exp(w^2)-w^2-1\big)dx.
$ In order to achieve this issue, we devoted ourself to defining an appropriate function $\varphi_d$ and computing $M[\varphi_d]$. We need some preparation at first.
\vskip0.1cm

  For any fixed $P\in \partial\Omega$, select the coordinate system with origin at $P$ and the inner normal to $\partial\Omega$ at $P$ is the positive $y$-axis. Then we introduce a diffeomorphism which straightens a boundary portion near $P\in \partial\Omega$. Since $P$ is the origin and the inner normal at $P$ is the positive $y$-axis, one can pick a smooth function $\varphi$ defined on $\{x\in\Omega\big||x|<\delta\}$ such that
  $$i)\ \varphi(0)=0\mbox{\ and\ } \varphi'(0)=0;$$
  $$ii)\  \partial\Omega\cap B_r(0)=\{(x,y)|y=\varphi(x)\}\mbox{\ and\ }\Omega\cap B_r(0)=\{(x,y)|y>\varphi(x)\},$$
  where $0<r<\delta$. For $z\in \mathbb{R}^2$ with $|z|\ll1$, define a function $x=\Phi(z)=(\Phi_1(z),\Phi_2(z))$ by
 \begin{equation}\label{equg.11}\begin{cases}
\Phi_1(z)=z_1-\varphi'(z_1),\\
\Phi_2(z)=z_2+\varphi(z_1).\\
\end{cases}\end{equation}
Then it follows from $\varphi'(0)=0$ that $D\Phi(0)=I$. As a consequence, there exists a converse mapping $z=\Phi^{-1}(x)=:\Psi(x)=(\Psi_1(x),\Psi_2(x))$ for $|z|<\delta'$. Let $0<3k<\delta$, then $x=\Phi(z)$ can be defined in $B_{\delta}(0)$ and $\Phi(B_{3k}^+(0))\subseteq\Omega$ where $B_r^+:=B_r(0)\cap\{y>0\}$. For any fixed $\rho>0$, a cut-off function $\xi_\rho:[0,+\infty)\rightarrow\mathbb{R}$ is denoted by
 \begin{equation}\label{equg.12}\xi_\rho(t)=\begin{cases}
1,\ \ \ \ \ \ \ \ \ 0\leq t\leq\rho,\\
2-\frac{t}{\rho},\ \ \ \rho< t\leq 2\rho,\\
0,\ \ \ \ \ \ \ \ \ 2\rho<t.
\end{cases}\end{equation}
Let $w=w(z)$ be a ground state solution of $-\Delta w+w=w(\exp(w^2)-1)$ and $w_*(z)=\xi_{\frac{k}{\sqrt d}}(|z|)w(z)$. Notice $w $ is radial and define
$$D_j:=\Phi(B_{jk}^+),\ \ j=1,2.$$
Based on the previous arguments, we are going to introduce the appropriate function $\varphi_d$ which is denoted by
\begin{equation}\label{equg.13}\varphi_d(x)=\begin{cases}
w_*(\frac{\Psi(x)}{\sqrt{d}}),\ \mbox{if}\  x\in D_2,\\
0,\ \ \ \ \ \ \ \ \ \ \mbox{if}\ x\in\mathbb{R}^2\backslash D_2.
\end{cases}\end{equation}
Before giving some asymptotic formulas on $M[\varphi_d]$, we give some lemmas about $w$ and $\Psi(x)$.
\begin{lemma}\label{NiW}(\cite{NiW} Lemma 3.3) Let $\gamma:=\frac13\int_{\mathbb{R}_+^2}w'(|z|)^2z_2dz$, then
\begin{equation}\label{equg.14}
\int_{\mathbb{R}_+^2}(\frac{\partial w}{\partial z_1})^2z_2=\gamma
\end{equation}
and
\begin{equation}\label{equg.15}
\int_{\mathbb{R}_+^2}(\frac{\partial w}{\partial z_2})^2z_2=2\gamma.
\end{equation}
Furthermore,
\begin{equation}\label{equg.16}
\int_{\mathbb{R}_+^2}\big(\frac12(|\nabla w|^2+w^2)-F(w)\big)z_2=2\gamma,
\end{equation}
where $F(v)=\frac12(\exp(v^2)-v^2-1)$.
\end{lemma}
\begin{lemma}\label{NiW1}(\cite{NiW} Lemma A.1)
If $|z|\rightarrow0$, then
\begin{equation}\label{equg.17}
det D\Phi(z)=1-\phi''(0)z_2+O(|z|^2)
\end{equation}
and
\begin{equation}\label{equg.18}
\Big|\frac{z}{|z|}D\Psi(\Phi(z))\Big|^2=1+z_2\frac{z_1^2}{|z|^2}\varphi_{11}+O(|z|^2).
\end{equation}
\end{lemma}
\vskip0.1cm

With these lemmas in mind, we are going to estimate the integral of $|\nabla \varphi_d|^2$ and get a generalized conclusion.

\begin{lemma}\label{3.4}
As $d\rightarrow0$, we have
\begin{equation}\label{equg.19}
d\int_\Omega|\nabla \varphi_d|^2dx=d\Big(\int_{\mathbb{R}_+^2}w'^2dz-\varphi''(0)\gamma d^\frac12
+O(d)\Big).
\end{equation}
Furthermore if $G:\mathbb{R}\rightarrow\mathbb{R}$ is locally H\"{o}lder continuous and $G(0)=0$, it follows that as $d\rightarrow0$,
\begin{equation}\label{equg.20}
\int_\Omega G(\varphi_d)dx=d\Big(\int_{\mathbb{R}_+^2} G(w)(1-
\varphi''(0)d^\frac12z_2)dz+O(d)\Big).
\end{equation}
\end{lemma}

Moreover, we obtain an estimate on $t_0$ which is the maximum point of $h_d$.

\begin{lemma}\label{lem3.5}
Let $h_d(t):=J_d(t\varphi_d)=\frac{t^2}{2}\int_\Omega\big(d|\nabla \varphi_d|^2+|\varphi_d|^2\big)dx-\frac12\int_\Omega\big(\exp(t^2\varphi_d^2)-
t^2\varphi_d^2-1\big)dx$. Then for $d$ sufficiently small, $h_d(t)$ has a maximum at $t_0(d)$ and
\begin{equation}\label{equg.21}
t_0(d)=1+\beta d^\frac12+o(d),
\end{equation}
where $\beta$ is a constant.
\end{lemma}
\begin{remark}
The proofs of lemma \ref{3.4} and \ref{lem3.5} are based on Lemma \ref{lem3.1}, \ref{NiW}, \ref{NiW1}. Since the proofs are similar to the proofs of Lemma 3.4 and 3.5 in \cite{NiW}, we omit the details.
\end{remark}

Based on the previous preparation, we are going to estimate $M[\varphi_d]$.
\begin{proposition}\label{pro3.2}
Let $\varphi_d$, $\gamma$ defined as before, then
\begin{equation}\label{equg.22}
M[\varphi_d]=d\Big(\frac12I(w)-\varphi''(0)\gamma d^\frac12+o(d^\frac12)\Big),
\end{equation}
where
\begin{equation}\label{ma.1}
I(w)=\frac12\int_{\mathbb{R}^2}(|\nabla w|^2+w^2)dx-\frac12\int_{\mathbb{R}^2}\big(\exp(w^2)-w^2-1\big)dx.
\end{equation}
\end{proposition}
\begin{proof}
Recall the definition of $t_0(d)$ and $M[\varphi_d]$, we have
\begin{equation}\label{equg.23}\begin{split}
M[\varphi_d]&=J_d(t_0(d)\varphi_d)\\
&=\frac12t_0(d)^2\int_\Omega \big(d|\nabla\varphi_d|^2+\varphi_d^2\big)dx-\frac12\int_\Omega\big(\exp(t_0(d)^2
\varphi_d^2)-t_0(d)^2
\varphi_d^2-1\big)dx\\
&=:I+II.
\end{split}\end{equation}
With the help of Lemma \ref{3.4}, we can derive that as $d\rightarrow0$,
\begin{equation*}\begin{split}
I&=\frac12t_0(d)^2d\big(\int_{\mathbb{R}_+^2}w'^2-\varphi''(0)\gamma
d^\frac12+O(d)\big)\\
&\ \ \ +\frac12t_0(d)^2d\big(\int_{\mathbb{R}_+^2}w^2(1-\varphi''(0)
d^\frac12z_2)dz+O(d)\big)\\
&=\frac12t_0(d)^2d\Big(\int_{\mathbb{R}_+^2}\big(w'^2+w^2-\varphi''(0)\gamma
d^\frac12-w^2\varphi''(0)d^\frac12z_2\big)dz+O(d)\Big).\\
\end{split}\end{equation*}
Then one can employ Lemma \ref{lem3.5} to derive that
\begin{equation}\begin{split}\label{equg.24}
I&=\frac12d\Big(\int_{\mathbb{R}_+^2}(w'^2+w^2)dz+o(d^\frac12)\Big)\\ &\ \ +\frac12d\big(2\beta\int_{\mathbb{R}_+^2}(w'^2+w^2)dz-
\varphi''(0)(\gamma+\int_{\mathbb{R}_+^2}w^2z_2dz)\big)
d^\frac12.
\end{split}\end{equation}
For $II$, it follows from the Taylor expansion and Lemma \ref{lem3.5} that
\begin{equation}\label{equg.25}\begin{split}
F(t_0(d)\varphi_d)&=
\exp((1+(\beta+o(1))d^\frac12)^2\varphi_d^2)-(1+(\beta+o(1))d^\frac12)^2\varphi_d^2-1\\
&=F(\varphi_d)+(\beta+o(1))d^\frac12\varphi_df(\varphi_d)+dg(\varphi_d,d),
\end{split}\end{equation}
where $F(v)=\frac12(\exp(v^2)-v^2-1)$, $f(v)=F'(v)$ and $|g(\varphi_d,d)|\lesssim\varphi_d^2$. By using Lemma \ref{estimate}, \ref{3.4} and inequality \eqref{equg.25}, one can derive that
\begin{equation}\label{equg.26}\begin{split}
II:&=\int_\Omega F(t_0(d)\varphi_d)dz\\
&=d\Big(\int_{\mathbb{R}^2_+}\big(
F(w)+(\beta+o(1))d^\frac12wf(w)\big)(1-\varphi''(0)d^\frac12z_2)dz+O(d)\Big)\\
&=d\Big(\int_{\mathbb{R}^2_+}\big(F(w)+d^\frac12(\beta wf(w)-\varphi''(0)z_2F(w))\big)dz+o(d^\frac12)\Big).
\end{split}\end{equation}
Combining \eqref{equg.23}, \eqref{equg.24} and \eqref{equg.26}, we have
\begin{equation}\label{equg.27}\begin{split}
M[\varphi_d]&=d\Big(\int_{\mathbb{R}^2_+}\frac12(w'^2+w^2)-F(w)dz
+d^\frac12\big(\beta\int_{\mathbb{R}^2_+}w'^2+w^2-wf(w)dz\\
&\ \ \ -\frac{\varphi''(0)}{2}(\gamma+\int_{\mathbb{R}^2_+}w^2z_2dz-
2\int_{\mathbb{R}^2_+}F(w)z_2dz)\big)+o(d^\frac12)\Big).
\end{split}\end{equation}
Notice that $w$ is a ground state solution of $-\Delta w+w=f(w)$ and $w$ is radial. Direct calculations show that
\begin{equation}\label{equg.27}\begin{split}
2\int_{\mathbb{R}^2_+}w'^2+w^2-wf(w)dz&=\int_{\mathbb{R}^2}|\nabla w|^2+w^2-wf(w)dz\\
&=\int_{\mathbb{R}^2}w(-\Delta w+w-f(w))dz\\
&=0.
\end{split}\end{equation}
Then it follows from Lemma \ref{estimate} that
\begin{equation}\label{equg.28}\begin{split}
\int_{\mathbb{R}^2_+}(\frac12w^2-F(w))z_2dz&
=\int_{\mathbb{R}^2_+}(\frac12|\nabla w|^2+\frac12w^2-F(w))z_2dz-\frac12
\int_{\mathbb{R}^2_+}|\nabla w|^2z_2dz\\
&=2\gamma-\frac32\gamma\\
&=\frac12\gamma.
\end{split}\end{equation}
This together with the definition of $I(w)$ and \eqref{equg.27} yields that
$$M[\varphi_d]=d\Big(\frac12I(w)-\varphi''(0)\gamma d^\frac12+o(d^\frac12)\Big).$$
Thus, we complete the proof of Proposition \ref{pro3.2}.
\end{proof}
\section{Proofs of Theorems \ref{theorem2.2} and \ref{theorem2.3}}
In this section, we focus on the shape of ground state solution $u_d$ around its condensation point $P_d$. Indeed, we show that $u_d$ exhibits "phenomenon of point condensation" in Theorem \ref{theorem2.2} and give a description of $u_d$ near $P_d$ in Theorem \ref{theorem2.3}.
 The proof of Theorem \ref{theorem2.2} is divided into three steps. Step 1 states that the maximum point $P_d$ is very close to the boundary, namely $d(P_d,\partial\Omega)=O(\sqrt d)$. $P_d$ is located on the boundary and $u_d$ has at most one local maximum point are discussed in Step 2 and Step 3.
The basic idea of the proof is to approximate $u_d$ around $P_d$ by a scaled positive radial solution. We apply energy threshold of the ground state solution $m_{d_j}<4d_j\pi$, a cut-off function, the concentration compactness principle for Trudinger-Moser inequality, regularity theory for elliptic equation and an accurate analysis on $m_d$ as $d\rightarrow0$ (see Proposition 3.8) to overcome the difficulty caused by the Trudinger-Moser growth.
\vskip0.2cm

\emph{Proof of Theorem \ref{theorem2.2}:} Suppose $u_d$ achieves its local maximum at $P_d\in\overline{\Omega}$. The proof is devided into three steps.
\medskip

\emph{Step 1.} For $0<d\ll1$, we claim that there exists a $C_*>0$ such that
\begin{equation}\label{equg2.2.1}
d(P_d,\partial\Omega)\leq C_*d^\frac12.
\end{equation}
Suppose \eqref{equg2.2.1} not hold, then there exists a sequence of $(d_j)_j$ satisfying $d_j\rightarrow0$ such that as $j\rightarrow+\infty$,
\begin{equation}\label{equg2.2.2}
\rho_j:=d_j^{-\frac12}d(P_{d_j},\partial\Omega)\rightarrow+\infty.
\end{equation}
Let $P_j=P_{d_j}$ and define a function $v_j:B_{\rho_j}\rightarrow\mathbb{R}$ by
\begin{equation}\label{equg2.2.3}
v_j(z):=u_{d_j}(P_j+d_j^\frac12z),\ \ \forall z\in B_{\rho_j}.
\end{equation}
 Through direct calculations, we see that $v_j$ satisfies the equation $$-\Delta v_j+v_j=v_j(e^{v_j^2}-1)\ \ \mbox{in}\ B_{\rho_j}.$$
 Then we split the proof of \eqref{equg2.2.1} into two parts. In the first part, we show that $\{v_j\}_j$ converges to $w$ which is a solution to $-\Delta w+w=f(w)$ up to a sequence. The second part talks about a lower bound of $m_{d_j}$. Since $u_{d_j}$ is a ground state solution, we have
\begin{equation}\begin{split}\label{elliptic1}
m_{d_j}&=\frac{1}{2}\int_{\Omega}u_{d_j}^2 (e^{u_{d_j}^2}-1)dx-\frac{1}{2}\int_{\Omega}\big(e^{u_{d_j}^2}-1-u_{d_j}^2\big)dx\\
&\geq\frac{1}{2}\int_{\Omega}u_{d_j}^2 (e^{u_{d_j}^2}-1)dx-\frac{1}{4}\int_{\Omega}u_{d_j}^2 (e^{u_{d_j}^2}-1)dx\\
&=\frac{1}{4}\int_{\Omega}u_{d_j}^2 (e^{u_{d_j}^2}-1)dx.
\end{split}\end{equation}
Then it follows from Lemma \ref{minimaxleval} that
\begin{equation}\begin{split}\label{elliptic2}
\int_{\Omega}u_{d_j}^2 (e^{u_{d_j}^2}-1)dx<4d_j\pi.
\end{split}\end{equation}
Since $G_d(u_{d_j})=0$, we derive that
\begin{equation}\begin{split}\label{elliptic3}
d_j\|\nabla u_{d_j}\|_2^2+\|u_{d_j}\|_2^2=\int_{\Omega}u_{d_j}^2 (e^{u_{d_j}^2}-1)dx<4d_j\pi.
\end{split}\end{equation}
Simple calculations give that
\begin{equation}\label{18.2}
\int_{B_{\rho_j}}\big(|\nabla v_j|^2+|v_j|^2\big)dx\leq\frac{1}{{d_j}}\Big({d_j}\|\nabla u_{d_j}\|_2^2+\|u_{d_j}\|_2^2\Big)<4\pi.
\end{equation}
Then there exists a subsequence (still denote it by $\{v_j\}_j$) such that
$$v_j\rightharpoonup w\mbox{\ in \ } W_{loc}^{1,2}(\mathbb{R}^2).$$
For any $0<R<\frac12\rho_j$, define a cut-off function
$\phi_R:[0,+\infty)\rightarrow\mathbb{R}$ by
 \begin{equation*}\phi_R(t)=\begin{cases}
1,\ \ \ \ \ \ \ \ \ 0\leq t\leq R,\\
2-\frac{t}{R},\ \ \ R< t\leq 2R,\\
0,\ \ \ \ \ \ \ \ \ 2R<t.
\end{cases}\end{equation*}
Thus, we get
\begin{equation}\label{xg.1}\begin{split}
&\lim_{j\rightarrow+\infty}\int_{B_{2R}}|\nabla ( (v_j-w)\phi_R)|^2dx\\
&\leq(1+\varepsilon)\lim_{j\rightarrow+\infty}\int_{B_{2R}}|\nabla (v_j-w)|^2\phi_R^2dx
+C_\varepsilon\lim_{j\rightarrow+\infty}\int_{B_{2R}}(v_j-w)^2|\nabla\phi_R|^2dx\\
&=(1+\varepsilon)\lim_{j\rightarrow+\infty} \int_{B_{2R}}|\nabla v_j|^2-|\nabla w|^2dx\\
&<4\pi,
\end{split}\end{equation}
where $\varepsilon$ is picked in such a way that $(1+\varepsilon)(\|\nabla v_j\|_2^2+\|v_j\|_2^2)<4\pi$. Choosing $p>1$ such that $(1+\varepsilon)p(\|\nabla v_j\|_2^2+\|v_j\|_2^2)<4\pi$, one can apply Trudinger-Moser inequality in $W^{1,2}_0(\Omega)$ to obtain that
\begin{equation}\label{xg.2}\begin{split}
\sup_j\int_{B_{R}}\exp(p(v_j-w)^2)dx\leq\sup_j\int_{B_{2R}}\exp(p\phi_R^2(v_j-w)^2)dx<C.
\end{split}\end{equation}
Picking $1<q<p$ and $(1+\varepsilon_0)^2=\frac pq$, one can derive that
\begin{equation}\label{xg.3}\begin{split}
&\sup_j\int_{B_{R}}\exp(qv_j^2)dx\\
&\leq\sup_j\int_{B_{R}}\exp(q(1+\varepsilon_0)(v_j-w)^2+c_{\varepsilon_0}w^2)dx\\
&\leq\sup_j\Big(\int_{B_{R}}\exp(q(1+\varepsilon_0)^2(v_j-w)^2)dx\Big)^\frac{1}{1+\varepsilon_0}
\Big(\int_{B_{R}}\exp (\frac {1+\varepsilon_0}{\varepsilon_0}c_{\varepsilon_0}w^2)dx\Big)^\frac{\varepsilon_0}{1+\varepsilon_0}\\
&\lesssim1.
\end{split}\end{equation}
Through H\"{o}lder inequality, we derive that there exists $a>1$ such that
\begin{equation}\begin{split}\label{con.8}
\sup_{j}\|v_j(e^{v_j^2}-1)\|_{L^a(B_{R})}\lesssim1,
\end{split}\end{equation}
which implies that
\begin{equation}\label{equg2.2.6}
\|v_j\|_{L^{\infty}(B_{R})}\leq C'.
\end{equation}
Therefore, we can apply the regularity theorem for Laplace equation (see \cite{GT}) to derive that
\begin{equation}\label{equg2.2.7}
\|v_j\|_{\mathcal{C}^{1,\theta}(\overline{B_{R}})}\leq C_1,\ \ j\geq j_m,
\end{equation}
where $\theta\in(0,1)$. Then it follows from interior Schauder estimate in $B_{\frac R2}$ that for $j\geq j_m$,
\begin{equation}\label{equg2.2.8}
\|v_j\|_{\mathcal{C}^{2,\theta}(\overline{B_{\frac R2}})}\leq C_2.
\end{equation}
Obviously, $\{v_j\}_j$ is uniformly bounded and equicontinuous in $\mathcal{C}^{2}(\overline{B_{\frac R2}})$ which implies that $\{v_j\}_j$ is a relatively compact set. Since $0<R<\frac{\rho_j}{2}$ is arbitrary and $\rho_j\rightarrow+\infty$ as $j\rightarrow+\infty$, we can select a subsequence of $\{v_j\}_j$ (for simplicity, we still denote it by $\{v_j\}_j$) such that
$$v_j\rightarrow w\ \ \mbox{in\ } \mathcal{C}^{2}_{loc}(\mathbb{R}^2).$$
Clearly, $w\in \mathcal{C}^{2}(\mathbb{R}^2)\bigcap W^{2,q}(\mathbb{R}^2)$ and $w$ is a solution of $-\Delta w+w=f(w)$. Then, we claim that
\begin{equation}\label{18.1}v_j(0)\geq \ln^\frac122.\end{equation}
We prove \eqref{18.1} by contradiction. Suppose $v_j(0)<\ln^\frac122$, then for $x\in B_{\rho_j}$ and $x$ sufficiently close to $0$,
$$\Delta v_j(x)=v_j(x)\big(\exp(v^2_j(x))-2\big)>0,$$
which is a contradiction with $\Delta v_j(0)\leq0$. Thus, one can get \eqref{18.1} which yields that $w(0)\geq0$ and $w\not\equiv0$. The strong maxmium principle results in $w>0$. With the help of Lemma \ref{estimate}, we derive that
\begin{equation}\label{equg2.2.9}
0<w(r)\leq C_0\exp(-\mu r),
\end{equation}
where $C_0>0$, $\frac12<\mu\leq1$.
Therefore, for any fixed $R\gg1$, we set
\begin{equation*}\label{equg2.2.10}
\varepsilon_R:=C_0\exp(-\frac{R}{2} ).
\end{equation*}
One can pick $j_R$ so large that for any $j\geq j_R$, there holds
\begin{equation}\label{equg2.2.11}
\rho_j\geq 4R,\ \ \|v_j-w\|_{C^2(\overline{B_{R}})}\leq\varepsilon_R.
\end{equation}
For the second part, we devoted ourselves to deriving a lower bound of $m_{d_j}$. We claim that for $j\geq j_R$,
\begin{equation}\label{equg2.2.12}
m_{d_j}\geq d_j^\frac12\Big(\int_{B_R}\big(\frac12wf(w)-F(w)\big)dz-C_3R^2\varepsilon_R\Big).
\end{equation}
Since $m_{d_j}=M[u_{d_j}]=J_{d_j}(u_{d_j})$,
direct calculations show that
\begin{equation}\label{equg2.2.13}
m_{d_j}= \int_{\Omega}\big(\frac12u_{d_j}f(u_{d_j})-F(u_{d_j})\big)dz.
\end{equation}
Recall the definition of $f$ and $F$, we have
\begin{equation}\begin{split}\label{equg2.2.14}
m_{d_j}&\geq \int_{|z-P_j|<d_j^\frac12R}\big(\frac12u_{d_j}f(u_{d_j})-F(u_{d_j})\big)dz\\
&=d_j\int_{|x|<R}\big(\frac12v_{j}(x)f(v_{j}(x))-F(v_{j}(x))\big)dx\\
&=d_j\Big(\int_{B_R}\big(\frac12wf(w)-F(w)\big)dx+E_j\Big),
\end{split}\end{equation}
where
$$E_j:=\int_{B_R}\Big(\big(\frac12v_{j}(x)f(v_{j}(x))-F(v_{j}(x))\big)-\big(\frac12wf(w)-F(w)\big)\Big)dx.$$
Based on the definition of $f$, one can make sense of inequalities \eqref{equg2.2.7} and \eqref{equg2.2.11} to get that for $j\geq j_R$,
\begin{equation}\begin{split}\label{equg2.2.15}
&|w(x)f(w(x))-v_{j}(x)f(v_{j}(x))|\\
&\leq|f(w(x))-f(v_{j}(x))|w(x)+f(v_{j}(x))|w(x)-v_{j}(x)|\\
&\leq C_4\varepsilon_R
\end{split}\end{equation}
and
\begin{equation}\begin{split}\label{equg2.2.16}
|F(w(x))-F(v_{j}(x))|\leq |f(w(x)+\theta^*(w(x)-v_{j}(x)))||w(x)-v_{j}(x)|\leq C_5\varepsilon_R,
\end{split}\end{equation}
where $\theta^*\in(0,1)$. Together with \eqref{equg2.2.15} and \eqref{equg2.2.16}, one can obtain that
$$|E_j|\leq(\frac12C_4+C_5)\varepsilon_R|B_R|,$$
which results in \eqref{equg2.2.12}. Then we focus on the relationship between $m_{d_j}$ and $I(w)$. Direct calculations show that
\begin{equation}\begin{split}\label{equg2.2.17}
\int_{B_R}\big(\frac12wf(w)-F(w)\big)dx&=I(w)-\int_{B_R^c}\big(\frac12wf(w)-F(w)\big)dx\\
&=I(w)-\int_{B_R^c}\frac12\big(w^2\exp(w^2)-\exp(w^2)+1\big)dx\\
&\geq I(w)-C_6\exp(-\mu  R),
\end{split}\end{equation}
where the last inequality comes from Lemma \ref{estimate}. Therefore \eqref{equg2.2.12} and \eqref{equg2.2.17} give that
\begin{equation}\label{equg2.2.18}
m_{d_j}\geq d_j^\frac12(I(w)-C\exp(-\eta  R)),
\end{equation}
where $C$ and $\eta$ are positive and independent of $j$ and $R$.
 Let $R$ be sufficiently large, we see that $m_{d_j}\geq \frac12d_j^\frac12I(w)$. Thanks to the definition of $M[\varphi_{d_j}]$, we see that $m_{d_j}\leq M[\varphi_{d_j}]$. Hence, $ \frac12d_j^\frac12I(w)\leq m_{d_j}\leq M[\varphi_{d_j}]$. Since $\partial\Omega$ is a compact smooth manifold without boundary, there is a $P\in\partial\Omega$ such that the mean curvature $H_p>0$ which implies $\varphi''(0)>0$.
 Then it follows from Proposition \ref{pro3.2} that $M[\varphi_{d_j}]<\frac12d_j^\frac12I(w)$ which contradicts with $\frac12d_j^\frac12I(w)\leq M[\varphi_{d_j}]$. Summarizing the above analysis, we see \eqref{equg2.2.1} holds.
 \medskip

 \emph{ Step 2.} In this step, we show that $P_d\in \partial\Omega$ for $0<d\ll1$. Suppose there exists a sequence $\{d_k\}_k$ which decreases and converges to $0$ such that $P_k:=P_{d_k}\in\Omega$. Since \eqref{equg2.2.1} holds and $\Omega$ is a bounded domain, we can see that up to a sequence, $P_k\rightarrow P\in\partial\Omega$ as $k\rightarrow+\infty$. 
 Define
 $$v_k(z):=u_{d_k}(P_k+d_k^\frac12z),\ \ \forall z\in \Omega_k,$$
 where
$$\Omega_k:=\{z|P_k+d_k^\frac12z\in \Omega\}.$$
Direct calculations show that $v_k$ is a weak solution of
\begin{equation}\label{18.3}\begin{cases}
-\Delta v_k+v_k=v_k(e^{v_k^2}-1)\ \ \mbox{in}\ \Omega_k,\\
\frac{\partial v_k}{\partial\nu}=0\ \ \ \ \ \ \ \ \ \ \ \ \ \ \ \ \ \ \ \ \ \ \ \ \mbox{on}\ \partial \Omega_k.\\
\end{cases}
\end{equation}
With similar progress of \eqref{elliptic1}, \eqref{elliptic2}, \eqref{elliptic3} and \eqref{18.2}, we have
\begin{equation}\label{18.4}
\int_{\Omega_k}\big(|\nabla v_k|^2+|v_k|^2\big)dx\leq\frac{1}{d_k}\Big(d_k\|\nabla u_{d_k}\|_2^2+\|u_{d_k}\|_2^2\Big)<4\pi.
\end{equation}
Since \eqref{equg2.2.1} holds, through rotation and translation, one can apply \eqref{equg2.2.1} and \eqref{18.2} to get that up to a sequence,
 $$v_k\rightharpoonup v_0 \mbox{\ in \ } W_{loc}^{1,2}(\mathbb{R}^2_+)\ \mbox{and\ }v_k\rightarrow v_0\mbox{\ a.e. in}\ \mathbb{R}^2_+,$$
 where $v_0$ is a solution of the following equation:
 \begin{equation}\label{18.5}\begin{cases}
-\Delta v_0+v_0=v_0(e^{v_0^2}-1)\ \ \mbox{in}\ \mathbb{R}^2_+,\\
\frac{\partial v_0}{\partial\nu}=0\ \ \ \ \ \ \ \ \ \ \ \ \ \ \ \ \ \ \ \ \ \ \ \ \mbox{on}\ \partial \mathbb{R}^2_+.\\
\end{cases}
\end{equation}
Through Fatou Lemma, we get
 \begin{equation}\begin{split}\label{18.5}
 I(v_0):&=\frac12\int_{\mathbb{R}^2_+}\big(|\nabla v_0|^2+|v_0|^2\big)dx-\frac12\int_{\mathbb{R}^2_+}F(v_0)dx\\
 &=\frac12\int_{\mathbb{R}^2_+}\big(v_0^2(e^{v_0^2}-1)\big)dx-
 \frac12\int_{\mathbb{R}^2_+}\big(e^{v_0^2}-1-v_0^2\big)dx\\
 &\leq\frac12\lim_{k\rightarrow+\infty}\int_{\Omega_k}
 \big(v_k^2(e^{v_k^2}-1)-(e^{v_k^2}-1-v_k^2)\big)dx\\
 &=:\lim_{k\rightarrow+\infty}I(v_k).
\end{split}\end{equation}
Through simple calculations, one can obtain that
$$m_{d_k}=J_{d_k}(u_{d_k})
=dI(v_k).$$
Then it follows from Proposition \ref{pro3.2} that
\begin{equation}\begin{split}\label{18.6}
m_{d_k}
&=dI(v_k)\\
&\leq d\{\frac12\int_{\mathbb{R}^2_+}\big(|\nabla v_0|^2+|v_0|^2-F(v_0)\big)dx-\varphi''(0)\gamma d^\frac12+O(d^\frac12)\}\\
&= d\{I(v_0)-\varphi''(0)\gamma d^\frac12+O(d^\frac12)\}\\
&\leq d I(v_0).
\end{split}\end{equation}
Combining \eqref{18.5} with \eqref{18.6}, we get
\begin{equation}\label{18.7}
\lim_{k\rightarrow+\infty}I(v_k)=I(v_0).
\end{equation}
With the help of Lemma \ref{lem3}, one can manage direct calculations to deduce that
\begin{equation}\label{18.8}
\lim_{k\rightarrow+\infty}\int_{\Omega_k}F(v_k)dx=\int_{\mathbb{R}^2_+}F(v_0)dx.
\end{equation}
Hence, it follows from \eqref{18.7} and \eqref{18.8} that
\begin{equation}\label{18.9}
\lim_{k\rightarrow+\infty}\int_{\Omega_k}\big(|\nabla v_k|^2+|v_k|^2\big)dx=\int_{\mathbb{R}^2_+}\big(|\nabla v_0|^2+|v_0|^2\big)dx,
\end{equation}
which implies
\begin{equation}\label{18.10}
v_k\rightarrow v_0\mbox{\ in\ }W^{1,2}_{loc}(\mathbb{R}^2_+).
\end{equation}
Manage the similar progress as Step 1, inequalities \eqref{xg.3}-\eqref{equg2.2.8}, we have
\begin{equation}\label{18.11}
v_k\in L_{loc}^\infty(\mathbb{R}^2_+)\mbox{\ and\ }v_k\rightarrow v_0\mbox{\ in\ } C_{loc}^2(\mathbb{R}^2_+).
\end{equation}
  Denote functions $y=\Psi(x)$ near $P$ and $\Phi=\Psi^{-1}(x)$ as before in an open set containing the closed ball $\overline{B_{2\kappa}},\kappa>0$. Note that $\Psi(x)$ straightens a boundary portion near $P$. Put $Q_k:=\Psi(P_k)\in B_\kappa^+$ for all $k\in \mathbb{N}^+$. Let
 \begin{equation}\label{equg2.2.19}
h_k(y):=u_{d_k}(\Phi(y)),\ \ \forall\ y\in\overline{B^+_{2\kappa}}
\end{equation}
and
\begin{equation}\label{equg2.2.20}
\widetilde{h}_k(y)=\begin{cases}
h_k(y),\ \ \ \ \ \ \ \ \ \mbox{if}\  y\in\overline{B^+_{2\kappa}},\\
h_k(y_1,-y_2),\ \ \mbox{if}\ y\in B^-_{2\kappa}.
\end{cases}\end{equation}
Denote a function $w_k(z)$ by
\begin{equation}\label{equg2.2.21}
w_k(z)=\widetilde{h}_k(Q_k+\sqrt{d_k}z),\ \forall z\in \overline{B_{\kappa\backslash\sqrt{d_k}}}.
\end{equation}
Let $Q_k=(q'_k,\alpha_k\sqrt{d_k})$. Then $\alpha_k>0$ and \eqref{equg2.2.1} yields that $\{\alpha_k\}$ is a bounded sequence. Since $\frac{\partial h_k}{\partial y_2}=0$ on $\{y_2=0\}$, one can see that
$$w_k\in C^2(\overline{B_{\kappa\backslash\sqrt{d_k}}}\backslash
\{z_2=-\alpha_k\})\cap C^1(\overline{B_{\kappa\backslash\sqrt{d_k}}}).$$
 For any $R>0$, there exists $K_R>0$ such that for $k>K_R$, $Rd_k^\frac12<\kappa$. Direct calculations show that for any $z\in B_R^+$,
 \begin{equation}\begin{split}\label{19.1}
 w_k(z)&=u_{d_k}(\Phi(Q_k+d_k^\frac12z))\\
 &=u_{d_k}(P_k+d_k^\frac12z+o(d_k^\frac12z)).\\
 \end{split}\end{equation}
 This together with $v_k(z)=u_{d_k}(P_k+d_k^\frac12z)$ and $v_k\rightarrow v_0$\ in $ C_{loc}^2(\mathbb{R}^2_+)$ yields that
  \begin{equation}\label{19.2}
 w_k\rightarrow w\ in\ C_{loc}^2(\mathbb{R}^2_+).
\end{equation}
Direct calculations show that $w_k$ satisfies the following equation:
\begin{equation}\label{equg2.2.21}
\sum\limits_{i,j=1}^{2}a_{ij}^k(z)\frac{\partial^2w_k}{\partial z_i\partial z_j}+d_k^\frac12\sum\limits_{j=1}^{2}b_j^k(z)\frac{\partial w_k}{\partial z_j}-w_k+f(w_k)=0,\ \ z\in B_{\kappa\backslash\sqrt{d_k}}\backslash
\{z_2=-\alpha_k\},
\end{equation}
where $a_{ij}^k(z)$ and $b_j^k(z)$ are piecewise functions which denoted by
\begin{equation*}
a_{ij}^k(z):=\begin{cases}
a_{ij}(Q_k+\sqrt{d_k}z),\ \ \ \ \ \ \ \ \ \ \ \ \ \ \ \ \ \ \ \ \ \ \ \ \ \ \ \ \ \ \ \ \ \ \ \mbox{if}\ z_2\geq-\alpha_k,\\
(-1)^{\delta_{i2}+\delta_{j2}}a_{ij}(q'_k+\sqrt{d_k}z_1,-(\alpha_k+z_2)\sqrt{d_k}),\ \mbox{if}\ z_2<-\alpha_k
\end{cases}\end{equation*}
and
\begin{equation*}
b_{j}^k(z):=\begin{cases}
b_{j}(Q_k+\sqrt{d_k}z),\ \ \ \ \ \ \ \ \ \ \ \ \ \ \ \ \ \ \ \ \ \ \ \ \ \ \ \ \ \ \ \ \ \ \ \ \ \mbox{if}\ z_2\geq-\alpha_k,\\
(-1)^{\delta_{j2}}b_{j}(q'_k+\sqrt{d_k}z_1,-(\alpha_k+z_2)\sqrt{d_k}),\ \ \ \ \ \ \ \mbox{if}\ z_2<-\alpha_k.
\end{cases}\end{equation*}
Throughout the definition of $a_{ij}^k(z)$ and  $b_{j}^k(z)$, $\delta_{ij}$ is the Kronecker symbol and
$$a_{ij}(y):=\frac{\partial\Psi_i}{\partial x_1}(\Phi(y))\frac{\partial\Psi_j}{\partial x_1}(\Phi(y))+\frac{\partial\Psi_i}{\partial x_2}(\Phi(y))\frac{\partial\Psi_j}{\partial x_2}(\Phi(y)),$$
$$b_j(y):=(\Delta \Psi_j)(\Phi(y)).$$
Then we focus on the Lipschitz continuity of $a_{ij}^k(z)$ and $b_j^k(z)$. For $a_{11}^k(z)$, $a_{22}^k(z)$ and $b_1^k(z)$, one can easily see that they are Lipschitz continuous in $\overline{B_{\kappa\backslash\sqrt{d_k}}}$ and their Lipschitz constants are uniformly bounded in $k$( indeed they depend on $\Psi$ and $\Phi$ which are independent of $k$). Direct calculations show $a_{12}^k(z_1,-\alpha_k)=a_{21}^k(z_1,-\alpha_k)=0$, then one can follow the similar line of (4.10) in \cite{LNT} to obtain that
$a_{12}^k(z_1,-\alpha_k)$ and $a_{21}^k(z_1,-\alpha_k)$ are also Lipschitz continuous in $\overline{B_{\kappa\backslash\sqrt{d_k}}}$ and their Lipschitz constants are uniformly bounded in $k$.
 For $b_2^k(z)$, we have $b_2^k(z)\frac{\partial w_k}{\partial z_2}$ is Lipschitz continuous since $\frac{\partial w_k}{\partial z_2}(z_1,-\alpha_k)=0$. Based on the definition of $w_k$ and $w_k\rightarrow w$ in $C^{2}(\mathbb{R}^2_+)$, we have
 $$w_k\in L^\infty_{loc}(\mathbb{R}^2).$$
Take $b_2^k(z)\frac{\partial w_k}{\partial z_2}$ as an inhomogeneous term, one can manage the same progress as Step 1 to deduce that up to a sequence,
$$w_k\rightarrow w \mbox{\ in\ }C_{loc}^2(\mathbb{R}^2),$$
which yields that
$$w\in C^2(\mathbb{R}^2)\cap W^{2,r}(\mathbb{R}^2).$$
Simple calculations show that
\begin{equation*}
\sum\limits_{i,j=1}^{2}a_{ij}(0)\frac{\partial^2w}{\partial z_i\partial z_j}-w+f(w)=0.
\end{equation*}
Note that $D\Psi(0)=\big(D\Phi(0)\big)^{-1}=I$, we have
\begin{equation*}
\Delta w-w+f(w)=0.
\end{equation*}
 Applying Lemma \ref{estimate}, one can derive that $|w(r)|\leq e^{-\theta r}$ for $r$ sufficiently large.
Let $R$ be a large number such that $R>\sup\limits_k\alpha_k$ and $\varepsilon_R$ as defined in Step 1. One can pick $K_R>0$ such that for $k>K_R$,
\begin{equation}\label{equg2.2.24}
\|w_k-w\|_{C^2(\overline{B_{4R}})}\leq \varepsilon_R.
\end{equation}
Next, we show that $w_k$ has only one local maximum point in $B_R$. As discussed in \eqref{18.1}, we see $w(0)=\lim_{k\rightarrow+\infty}w_k(0)\geq\ln^\frac122$
which yields
 $$w''(0)<0.$$

 Therefore, we can pick two numbers $a$ and $b$ satisfying $0<a<b$ such that (i) $w''(r)<0$, $\forall r\in[0,a]$ and (ii) $w(b)<\ln^\frac122$. Since $w$ is strictly decreasing, then $c_*:=\min\{|w'(r)||r\in[a,b]\}>0$. For $|z|\in[a,b]$, one can employ \eqref{equg2.2.24} to derive that for any $\varepsilon_R<c_*$,
$$|\nabla w_k(z)|\geq |\nabla w(z)|-|\nabla w_k(z)-\nabla w(z)|\geq c_*-\varepsilon_R>0,$$
which implies that $w_k(z)$ is decreasing in $[a,b]$. Then one can apply Lemma 4.2 (see \cite{NiW}) in the ball $\overline{B_a}$ and obtain that the origin is the unique local maximum point of $w_k$. This together with  $w_k(z)$ is decreasing in $[a,b]$, we see $z=0$ is the unique local maximum point of $w_k$ in $B_b$. For $z_k\in B_R\backslash B_b$, one can choose $\varepsilon_R<\ln^\frac122-w(b)$ to get $w_k(z)\leq w(z)+\varepsilon_R<\ln^\frac122$. Therefore there is no local maximum point in $B_R\backslash B_b$. As a consequence, $z_k=0$. Hence, $\widetilde{h}_k$ has the only local maximum point and $\alpha_k=0$. Step 2 is finished.
 \vskip0.2cm

 \emph{ Step 3.} We will prove that $u_d$ has at most one local maximum point. Otherwise, there is a decreasing sequence $\{d_k\}$ converges to $0$ such that $u_{d_k}$ achieves its local maximum at $P_k$ and $P'_k$. From the arguments of Step 1 and Step 2, we see $P_k\in \partial\Omega$, $P'_k\in \partial\Omega$. Since $w_k$ has a unique local maximum point in $B_R$, we have
 $\frac{|P_k-P'_k|}{\sqrt {d_k}}\rightarrow+\infty$.
 \vskip 0.1cm

 As discussed in Step 2, we introduce $y=\Psi(x)$ which defined near the accumulation point of $P_k$ and denoted $v_k$, $\widetilde{h}_k$ and $w_k$ as before. Similar arguments as Step 2 yield that up to a sequence, $w_k\rightarrow w$ in $C_{loc}^2(\mathbb{R}^2)$, $w\in C^2(\mathbb{R}^2)\cap W^{1,2}(\mathbb{R}^2)$.
For $R>0$, we define $\varepsilon_R=C_0\exp(-\frac{R}{2} )$ and direct calculations give that $$\|w_k-w\|_{C^2(\overline{B_{4R}})}\leq \varepsilon_R.$$
Then we give an estimate of $m_{d_k}$, we write
\begin{equation}\begin{split}\label{equg2.2.27}
m_{d_k}&=\int_\Omega\big(\frac12u_{d_k}f(u_{d_k})-F(u_{d_k})\big)dx\\
&=\int_{\Phi(B^+_{R\sqrt{d_k}})}\big(\frac12u_{d_k}f(u_{d_k})-F(u_{d_k})\big)dx\\
&\ \ +\int_{\Omega\backslash\Phi(B^+_{R\sqrt{d_k}})}\big(\frac12u_{d_k}f(u_{d_k})-F(u_{d_k})\big)dx\\
&=:I_1+I_2.
\end{split}\end{equation}
Through similar calculations as in \eqref{equg2.2.12}, we derive that
\begin{equation}\begin{split}\label{equg2.2.25}
I_1&\geq d_k\Big(\int_{B^+_R}\big(\frac12wf(w)-F(w)\big)dx-C_1R^2\varepsilon_R
\Big)\\
&\geq d_k\big(\frac12I(w)-C_1R^2\varepsilon_R\big).
\end{split}\end{equation}
As for $I_2$, since $\frac{|P_k-P'_k|}{\sqrt {d_k}}\rightarrow+\infty$, we see $B^+_{R\sqrt{d_k}}(P'_k)\cap\Omega \subset \Omega\backslash\Phi(B^+_{R\sqrt{d_k}})$. Then it follows from Harnack inequality that $\frac12u_{d_k}f(u_{d_k})-F(u_{d_k})\geq \eta_0$ on $B^+_{R\sqrt{d_k}}(P'_k)\cap\Omega$. Therefore
\begin{equation}\label{equg2.2.26}
I_2\gtrsim d_k.
\end{equation}
Combining \eqref{equg2.2.27}, \eqref{equg2.2.25}, \eqref{equg2.2.26} and Lemma \ref{estimate}, one can apply Lemma \ref{estimate} to obtain that
\begin{equation}\label{equg2.2.28}
m_{d_k}\geq d_k\Big(\frac12I(w)+C-C_0\exp(-\mu R)\Big).
\end{equation}
With the help of Proposition \ref{pro3.2}, one can pick $P\in \partial\Omega$ such that $H_{\partial\Omega}(P)>0$ to derive that for ${d_k}$ sufficiently small,
\begin{equation*}
m_{d_k}\leq M[\varphi_{d_k}]<{d_k}\frac12I(w),
\end{equation*}
which contracts with \eqref{equg2.2.28}. Hence, $u_d$ has at most one local maximum point which completes the proof of Theorem \ref{theorem2.2}.
\medskip

Then we show the shape of ground state solution $u_d$ and give the proof of Theorem \ref{theorem2.3}:
\vskip0.2cm

\emph{Proof of Theorem \ref{theorem2.3}:} For any arbitrary decreasing sequence such that $\lim\limits_{k\rightarrow+\infty}d_k=0$, define $v_k$, $\widetilde{h}_k$ and $w_k$ as in Step 2, Theorem \ref{theorem2.2}. Notice that $Q_k=0$ and one can manage the similar progress as Step 2 to derive that up to a sequence,
\begin{equation*}
w_{k_n}\rightarrow w\mbox{\ in\ }C_{loc}^2(\mathbb{R}^2),
\end{equation*}
where $w\in C_{r}^2(\mathbb{R}^2)\cap W^{2,2}(\mathbb{R}^2)$ and satisfies $\Delta w-w+f(w)=0$. Then we show $w$ is a ground state solution. Suppose it is not true, there exists a positive radial $w_0$ satisfies $\Delta w_0-w_0+f(w_0)=0$ such that $I(w)>I(w_0)$. Through similar calculations as \eqref{equg2.2.28}, one can derive that
\begin{equation}\label{equg2.2.29}
m_{d_{k_n}}\geq d_{k_n}\Big(\frac12I(w)-C_0\exp(-\mu R)\Big).
\end{equation}
As for $w_0$, one can define $\varphi_d$ using $w_0$ instead of $w$ and manage same argument as Proposition \ref{pro3.2} to get that
\begin{equation}\label{equg2.2.30}
m_{d_{k_n}}\leq M[\varphi_{d_{k_n}}]<d_{k_n}\frac{I(w_0)}{2}.
\end{equation}
This together with $I(w)>I(w_0)$ yields a contradiction with \eqref{equg2.2.29} provided $R$ sufficiently large. Thus $w$ is a ground state solution. Since the uniqueness assumption, we see that $w_k\rightarrow w$. Therefore, the function $w_d(z)$ denoted by
\begin{equation}\label{equg2.2.31}
w_d(z):=\begin{cases}
u_d(\Phi(\sqrt dz),\ \ \ \ \ \ \ \ \ \ \ \ \ \  \mbox{for}\ z_2\geq0,\\
u_d(\Phi(\sqrt dz_1,-\sqrt dz_2),\ \ \ \mbox{for}\ z_2<0
\end{cases}\end{equation}
converges to $w$ in $C_{loc}^2(\mathbb{R}^2)$.
\vskip 0.1cm

Based on Lemma \ref{estimate}, we see $w(r)\leq C_0\exp(-\frac{r}{2})$. Let $R=2\log(\frac{C_0}{\varepsilon})$, then $\varepsilon=C_0\exp(\frac{R}{2})$. For $\beta>0$, then exists a $d_{\varepsilon,\beta}>0$ such that
\begin{equation}\label{equg2.2.32}
\|w_d-w\|_{\mathbb{C}^2(\overline{B}_R)}\leq \beta \varepsilon,
\end{equation}
if $0<d<d_{\varepsilon,\beta}$. To prove (i), one can pick $\Omega_d^{(\varepsilon)}:=\Phi(B_{R\sqrt d}^+)$ to get it.
\vskip 0.1cm

In order to show (ii), we note that $D\Psi$, $D^2\Psi$ are uniformly bounded. Hence, direct calculations show that
\begin{equation}\begin{split}\label{equg2.2.33}
\|u_d(x)-w(\Psi(x)/\sqrt d)\|_{C^2(\overline{\Omega}_d^{(\varepsilon)})}&\leq
C\|w_d-w\|_{C^2(\overline{B}_R)}\\
&\leq C \beta \varepsilon\\
&= \varepsilon,
\end{split}\end{equation}
where $C_*>0$ depending only on $\Psi$ and $\beta=C_*^{-1}$. Therefore, (ii) is finished.
\vskip 0.1cm

For (iii), notice that $w_d(z)\leq (1+\beta)\varepsilon$ for $\frac{R}{2}\leq |z|\leq R$. As a consequence, we have $u_d(x)\leq (1+\beta)\varepsilon$ for $x\in \partial\Omega_d^{(\varepsilon)} \cap \Omega$. Then through Theorem \ref{theorem2.2}, we see $\{x\in\Omega|u_d(x)>(1+\beta)\varepsilon\}\subset \Omega_d^{(\varepsilon)} $. Therefore, $u_d(x)\leq (1+\beta)\varepsilon$ for $x\in\Omega\setminus \Omega_d^{(\varepsilon)} $. Then define $v_d$ by
\begin{equation*}
v_d(y):=\begin{cases}
u_d(\Phi(y)),\ \ \ \ \ \ \ \ \ \ \ \ \ \  \mbox{for}\ y_2\geq0,\\
u_d(\Phi(y_1,-y_2),\ \ \ \ \ \ \ \ \mbox{for}\ y_2<0,
\end{cases}
\end{equation*}
where $\Phi$ can be a diffeomorphism which straightens a boundary portion at each point of $\partial\Omega$.
Direct calculations give that
\begin{equation*}
d\Big(\sum\limits_{i,j=1}^2a_{ij}\frac{\partial^2 v_d}{\partial y_i \partial y_j}+\sum\limits_{j=1}^{2}b_j\frac{\partial v_d}{\partial y_j}\Big)-\Big(1-\frac{f(v_d)}{v_d}\Big)v_d=0.
\end{equation*}
Since $\lim\limits_{s\rightarrow0}\frac{f(s)}{s}=0$, then $1-\frac{f(v_d)}{v_d}>0$ provided the boundary portion at the point on $\partial\Omega\cap\partial\Omega_d^{(0)} $. Then, one can use Lemma 4.2 in \cite{Fi} to achieve $v_d$ in a neibourhood of $\partial\Omega\cap\partial\Omega_d^{(0)} $ and $u_d$ in $\Omega_d^{(0)}$ satisfying (iii).
\vskip0.2cm

\section{Proof of Theorem \ref{thm.1}}
In this section, we give the proof of Theorem \ref{thm.1}. The following lemma plays a key role in managing the rotating plane method to prove Theorem \ref{thm.1}.
\begin{lemma}\label{lem1.2}
Assume $u_d$ is a ground state solution of equation \eqref{equg.1} and $P_d$ is a maximum point of $u_d$. Then $P_d\neq0$ and $u_d$ is axially symmetric with respect to the line $\overrightarrow{OP_d}$. Moreover, if we assume $P_d$ located on the positive $x_2$-axis, then we have
\begin{equation}\label{anix.1}
x_1\frac{\partial u_d}{\partial x_2}-x_2\frac{\partial u_d}{\partial x_1}>0\ \mbox{for}\ x_1>0.
\end{equation}
\end{lemma}
\begin{proof}
We split the proof into four steps.
\vskip0.1cm

\emph{Step 1}. We claim that $u_d$ is not radially symmetric and prove it by contradiction. Assume $u_d$ is radially symmetric, then the first eigenfunction $\phi_1$ for the linearized equation is radially symmetric. Since $\frac{\partial u_d}{\partial x_1}(x)=u_d'(|x|)\frac{x_1}{|x|}$ and $\phi_1$ is radially symmetric, one can easily get that $\frac{\partial u_d}{\partial x_1}$ and $\phi_1$ are orthogonal in $L^2(\mathbb{R}^n)$. Through boundary condition, for $x\in \partial D$, we have $\frac{\partial u_d}{\partial x_1}(x)=u_d'(1)\frac{x_1}{|x|}=0$. Since $u_d$ is a ground state solution of equation \eqref{equg.1}, it is not difficult to check that the second eigenvalue $\mu_2$ of the linearized equation \eqref{equg.1} is nonnegative. Hence it follows that
\begin{equation*}\begin{split}
0&=\int_{D}\Big(d\Big|\nabla\frac{\partial u_d}{\partial x_1}\Big|^2+\Big|\frac{\partial u_d}{\partial x_1}\Big|^2-(e^{u^2}+2u^2e^{u^2}-1)\Big|\frac{\partial u_d}{\partial x_1}\Big|^2 \Big)dx\\
&\geq \inf_{\varphi\bot\phi_1}\int_{D}\Big(d|\nabla\varphi|^2+|\varphi|^2-(e^{u^2}+2u^2e^{u^2}-1)\varphi^2 \Big)dx\\
&\geq0.
\end{split}\end{equation*}
Therefore $\frac{\partial u_d}{\partial x_1}$ achieves the infimum and satisfies the Neumann condition for $x\in \partial D$ which implies that $u_d"(1)=0$. Together with $u_d'(1)=0$ and the uniqueness of ODEs, one can see that $u_d(|x|)\equiv\ln^\frac122$. This contradicts with $u_d$ being a nonconstant solution. Hence, $u_d$ is not radially symmetric.
\medskip

\emph{Step 2}. For $x\in D^+$, let $w(x)=u(x)-u(x^-)$, $x^-=(x_1,-x_2)$. Then we prove that for $x\in D^+$, $w(x)\geq0$ or $w(x)\leq0$. Assume it is not true, define $\Omega_+$ and $\Omega_-$ as follows which are nonempty:
\begin{equation*}
\Omega_+:=\{x\in D^+|w(x)>0\},\ \Omega_-=\{x\in D^+|w(x)<0\}.
\end{equation*}
Through direct calculations, $w$ satisfies
\begin{equation}\label{anix.2}\begin{cases}
\Delta w+\frac1d(c_0(x)-1)w=0\ \ \mbox{in}\ D,\\
w(x)=0\ \ \ \mbox{on}\ x_n=0,\ \ \  \frac{\partial w}{\partial\nu}=0\ \ \ \mbox{on}\ \partial D,\\
\end{cases}\end{equation}
where $c_0(x)=\frac{u(x)(e^{u^2(x)}-1)-u(x^-)(e^{u^2(x^-)}-1)}{u(x)-u(x^-)}.$
Denote a function $v$ by
\begin{equation}\label{anix.3}
v(x)=\begin{cases}
w(x),\ \ \ \ \ \mbox{if}\ x\in \Omega_+,\\
cw(x^-),\ \ \mbox{if}\ x\in\Omega_-^*:=\{x^-|x\in\Omega_-\},\\
0,\ \ \ \ \ \ \ \ \ \ \mbox{otherwise}.
\end{cases}\end{equation}
One can pick $c<0$ such that
$$\int_{D}v(x)\phi_1(x)dx=0,$$
where $\phi_1(x)>0$ is the first eigenfunction of linearized equation \eqref{equg.1}:
\begin{equation}\label{anix.4}\begin{cases}
-d\Delta \phi_1+(2u^2e^{u^2}+e^{u^2}-2)\phi_1+\mu_1\phi_1=0\ \ \mbox{in}\ D,\\
\frac{\partial \phi_1}{\partial\nu}=0\ \ \ \ \ \ \ \ \ \ \ \ \ \ \ \ \ \ \ \ \ \ \ \ \ \ \ \ \ \ \ \ \ \ \ \ \ \ \ \ \ \ \ \ \ \ \ \mbox{on}\ \partial D.\\
\end{cases}\end{equation}
Based on the definition of $v(x)$ and equation \eqref{anix.2}, one can derive the following inequality through direct calculations:
\begin{equation}\label{anix.5}
-v(x)(d\Delta v+(2u^2e^{u^2}+e^{u^2}-2)v)\begin{cases}
<0,\ \ \ \ \ \mbox{if}\ x\in \Omega_+,\\
<0,\ \ \ \ \ \mbox{if}\ x\in\Omega_-^*,\\
=0,\ \ \ \ \ \mbox{otherwise}.
\end{cases}\end{equation}
As a consequence, we have
\begin{equation}\label{anix.6}
\int_{D}-v(x)(d\Delta v+(2u^2e^{u^2}+e^{u^2}-2)v)dx<0.
\end{equation}
Since $\int_{D}v(x)\phi_1(x)dx=0$, one can apply $\mu_2\geq0$ to obtain that
\begin{equation}\begin{split}\label{anix.7}
&\int_{D}-v(x)(d\Delta v+(2u^2e^{u^2}+e^{u^2}-2)v)dx\\
&=\int_{D}d|\nabla v|^2+(2u^2e^{u^2}+e^{u^2}-2)v^2dx\\
&\geq0,
\end{split}\end{equation}
which contracts with \eqref{anix.6}. Therefore, $w(x)\geq0$ or $w(x)\leq0$ for $\in D^+$.
\medskip

\emph{Step 3}. $P_d\neq0$ and if we suppose $P_d$ located on the positive $x_2$-axis, then
\begin{equation}\label{anix.8}
w(x)=u(x)-u(x^-)>0.
\end{equation}
We prove $P_d\neq0$ by contradiction. If $P_d=0$, we claim that $w(x)\equiv0$. Otherwise, for $x\in D^+$, $w(x)\not\equiv0$. Thanks to Step 2 and the strong maximum principle, we have $w(x)>0$ for any $x\in D^+$. With the help of Hopf lemma, we get
$$-\frac{\partial w}{\partial x_2}(0)=-2\frac{\partial u}{\partial x_2}(0)<0.$$
However, $P_d$ is a maximum point yields that $\frac{\partial u}{\partial x_2}=0$. Thus, we get a contradiction and $w(x)\equiv0$. Managing the similar progress on $x_1$, one can obtain the radial symmetry of $u_d$ which contradicts with Step 1. Hence, $P_d\neq0$. Without loss of generality, we assume $P_d$ is located on the positive $x_2$-axis.\\
To prove \eqref{anix.8}, we redenote $\Omega_+$ and $\Omega_-$ as
$$\Omega_+:=\{x\in D^+|w(x)>0\},\ \Omega_-:=\{x\in D^+|w(x)<0\}.$$
Suppose $\Omega_+$ is empty, then we have $w(x)\leq 0$ on $D^+$. This together with $w(P_d)\geq0$ yields that $ w(P_d)=0$. Thanks to the strong maximum principle, one can easily get that $w(x)\equiv0$ which implies that $0$ is a maximum point. However we already have $P_d\neq0$. Therefore the assumption fails and $\Omega_+$ is not empty. Then one can manage the similar proof as step 2 to get $\Omega_-$ is empty. Therefore inequality \eqref{anix.8} is established.
\medskip

\emph{Step 4}: In this step, we apply the method of rotating planes to prove that $u_d$ is axially symmetric with respect to the line $\overrightarrow{OP_d}$. Without loss of generality, we assume $P_d$ located on the positive $x_2$-axis. For any $\theta\in [0,\frac{\pi}{2})$, $l_\theta$ denotes the line
$$\{(t\cos\theta,t\sin\theta)\}.$$
Let $x^\theta$ be the reflection point of $x$ with respect to $l_\theta$, $\Sigma_\theta$ be the component which contains $P_d$. Define a function $v_\theta$ on $\Sigma_\theta$ by
$$v_\theta(x)=u_d(x)-u_d(x^\theta).$$
Direct calculations show that \begin{equation}\label{anix.10}\begin{cases}
\Delta v_\theta(x)+\frac1d(c_\theta(x)-1)v_\theta(x)=0\ \ \mbox{in}\ \Sigma_\theta,\\
v_\theta(x)=0\ \ \ \mbox{on}\ l_\theta,\ \ \  \frac{\partial v_\theta}{\partial\nu}=0\ \ \ \mbox{on}\ \partial D\bigcap \overline{\Sigma_\theta},\\
\end{cases}\end{equation}
where $$c_\theta(x)=\frac{u(x)(e^{u^2(x)}-1)-u(x^\theta)(e^{u^2(x^\theta)}-1)}{u(x)-u(x^\theta)}.$$
Denote
$$\theta_0:=\sup\{\theta| v_{\widetilde{\theta}}(x)\geq0,\  \forall x\in\Sigma_{\widetilde{\theta}}, 0\leq \widetilde{\theta}\leq \theta<\frac{\pi}{2}\}.$$
Then we will show $\theta_0=\frac{\pi}{2}$ and argue it by contradiction. Assume $\theta_0<\frac{\pi}{2}$, one can get that $v_{\theta_0}(x)\geq0$ for $x\in\Sigma_{\theta_0}$ through continuity. Notice that $u_d$ is a nonconstant solution, one can derive that $v_{\theta_0}(P_d)>0$. Since if $v_{\theta_0}(P_d)=0$ and $\theta_0<\frac{\pi}{2}$, one can apply the strong maximum principle to obtain that $v_\theta(x)\equiv0$ and $u_d$ is radial.

Thanks to the fact $v_{\theta_0}(x)\geq0$ for $x\in\Sigma_{\theta_0}$ and $v_{\theta_0}(P_d)>0$, one can employ the Hopf lemma to derive that $v_{\theta_0}(x)>0$ for $x\in\overline{\Sigma_\theta}\backslash l_{\theta_0}$ and $\frac{\partial v_{\theta_0}}{\partial \nu}(x)<0$ for $x\in l_{\theta_0}\backslash \partial D$. Choosing a sequence of $\theta_j>\theta_0$ converges to $\theta_0$ such that
$$v_{\theta_j}(x_j)=\inf_{\overline{\Sigma_\theta}}v_{\theta_j}(x)<0.$$
Up to a sequence, $x_j$ converges to $x_0$. One can easily get that $v_{\theta_0}(x_0)=0$ and $\nabla v_{\theta_0}(x_0)=0$. Thus, $x_0\in l_{\theta_0}\bigcap \partial D$. Let $e_{1,j}$ be the base of $l_{\theta_j}$ and $e_1=\lim_{j\rightarrow+\infty} e_{1,j}$, then
$e_{1}$ is the base for $l_{\theta_0}$. Note that $v_{\theta_0}\equiv0$ on $l_{\theta_0}$, one can obtain that
\begin{equation*}
D_{e_1}D_{e_1}v_{\theta_0}(x_0)=0.
\end{equation*}
For $x_j\in \overline{\Sigma_\theta}\cap D$, it is easy to see that
\begin{equation}\label{anix.11}
D_{e_{1,j}}v_{\theta_j}(x_j)=0\ \mbox{and}\ D_{e_{1,j}}v_{\theta_j}(\hat{x}_j)=0,
\end{equation}
where $\hat{x}_j$ is the projection of $x_j$ on $l_{\theta_j}$. If $x_j\in \overline{\Sigma_\theta}\cap \partial D$, one can apply the fact $e_{1,j}$ is perpendicular to the outnormal of $\partial D$ and tangent to $\partial D$ at $x_j$ to derive inequality \eqref{anix.11}. Following from the mean value theorem, we have
\begin{equation*}
D_{e_2}D_{e_1}v_{\theta_0}(x_0)=0.
\end{equation*}
Combining this with equation \eqref{anix.11}, one can get $D_{e_2}D_{e_{2}}v_{\theta_0}(x_0)=0$. Since the Hessian of $v_{\theta_0}$ at $x_0$ vanishes which contradicts with $S$-Lemma in \cite{GN}. Hence, we derive that $\theta_0=\frac{\pi}{2}$ and $u_d(x)\geq u_d(x^\frac{\pi}{2})$. If we rotate planes in the opposite direction, then we have $u_d(x)\leq u_d(x^\frac{\pi}{2})$ which yields that $u_d(x)$ is symmetric with respect to $x_1$. Thus axial symmetry follows. For inequality \eqref{anix.1}, one can apply $\frac{\partial v_\theta}{\partial\nu}(x)<0$ on $l_\theta$ to derive it. Therefore, we complete the proof.
\end{proof}
\emph{Proof of Theorem \ref{thm.1}:} Without loss of generality, one can assume that $P_d$ is located on the positive $x_2$-axis. Since $u_d$ is axially symmetric with respect to the $x_2$-axis, one can apply direct calculations to know that
\begin{equation}\label{anix.13}
\sum_{j=1}^{2}\Big\{x_j^2\frac{\partial u_d}{\partial x_2}-x_2x_j\frac{\partial u_d}{\partial x_j}\Big\}
\end{equation}
is also axially symmetric with respect to the $x_2$-axis. Combining this with inequality \eqref{anix.1}, one can derive that for $x\in D\setminus \{x_1=0\}$, there holds
\begin{equation}\label{anix.14}
\sum_{j=1}^{2}\Big\{x_j^2\frac{\partial u_d}{\partial x_2}-x_2x_j\frac{\partial u_d}{\partial x_j}\Big\}>0.
\end{equation}
For $x\in \partial D$ satisfies $\frac{\partial u_d}{\partial x_2}(x)\leq0$, one can combine the Neumann condition to derive that
\begin{equation}\label{anix.15}
0=(-x_2)\frac{\partial u_d}{\partial \nu}(x)=(-x_2)(x\cdot \nabla u_d(x))\geq\sum_{j=1}^{2}\Big\{x_j^2\frac{\partial u_d}{\partial x_2}-x_2x_j\frac{\partial u_d}{\partial x_j}\Big\}.
\end{equation}
With the help of inequality \eqref{anix.14}, we know that $x$ must be $(0,\pm1)$ and $\frac{\partial u_d}{\partial x_2}(x)=0$ when $x=(0,\pm1)$. Therefore
\begin{equation}\label{anix.16}
\frac{\partial u_d}{\partial x_2}(x)>0,\ \ x\in \partial D\setminus (0,\pm1).
\end{equation}
Then, we claim that
\begin{equation}\label{anix.17}
\frac{\partial u_d}{\partial x_2}(x)>0,\ \ x\in D^-.
\end{equation}
Take the partial derivative of equation \eqref{equg.1}, one can obtain that
\begin{equation}\label{anix.18}\begin{cases}
d\Delta\frac{\partial u_d}{\partial x_2} +(e^{u_d^2}+2u_d^2e^{u_d^2}-2)\frac{\partial u_d}{\partial x_2}=0\ \ \mbox{in}\ D^-,\\
\frac{\partial u_d}{\partial x_2}(x)>0\ \ \ \ \ \ \ \ \ \ \ \ \ \ \ \ \ \ \ \ \ \ \ \ \ \ \ \ \ \ \ \ \ \mbox{on}\ \partial D^-\setminus\{(0,0,\cdots,-1)\}.\\
\end{cases}\end{equation}
Define $\Omega_-:=\{x\in D^-:\frac{\partial u_d}{\partial x_2}<0\}$ and $\Omega_-^-:=\{x\in D^-:x^-=(x_1,-x_2)\in \Omega_-\}$. Assume $\Omega_-$ is not empty. Denote a function $v$ by
\begin{equation}\label{anix.19}
v(x)=\begin{cases}
\frac{\partial u_d}{\partial x_2}(x),\ \ \ \ \ \mbox{if}\ x\in \Omega_-,\\
c\frac{\partial u_d}{\partial x_2}(x^-),\ \ \mbox{if}\ x\in\Omega_-^-,\\
0,\ \ \ \ \ \ \ \ \ \ \ \ \mbox{otherwise},
\end{cases}\end{equation}
where $c<0$ is picked in such a way that
$$\int_{D}v(x)\phi_1(x)dx=0,$$
where $\phi_1(x)>0$ is the first eigenfunction of \eqref{anix.4}. Through the Step 3 of Lemma \ref{lem1.2}, we know that for $x\in \Omega_-^-$, there holds $u_d(x)\geq u_d(x^-)$ which yields that
\begin{equation}\begin{split}\label{anix.20}
&\int_{\Omega_-^-}d|\nabla v|^2+(2-e^{u_d^2}-2u_d^2e^{u_d^2})v^2(x)dx\\
&=c^2\int_{\Omega_-}d|\nabla v|^2+(2-e^{u_d^2(x^-)}-2u_d^2e^{u_d^2(x^-)})v^2(x)dx\\
&<c^2\int_{\Omega_-}d|\nabla v|^2+(2-e^{u_d^2(x)}-2u_d^2e^{u_d^2(x)})v^2(x)dx\\
&=0.
\end{split}\end{equation}
However, according to the nonnegativity of the second eigenvalue $\mu_2$ of the linearized equation \eqref{equg.1}, we also have
\begin{equation}\begin{split}\label{anix.21}
&\int_{\Omega_-^-}d|\nabla v|^2+(2-e^{u_d^2}-2u_d^2e^{u_d^2})v^2(x)dx\\
&=\int_{D}d|\nabla v|^2+(2-e^{u_d^2}-2u_d^2e^{u_d^2})v^2(x)dx\\
&\geq0,
\end{split}\end{equation}
which is a contradiction with inequality \eqref{anix.20}. Hence, inequality \eqref{anix.17} holds.

Based on inequalities \eqref{anix.16} and \eqref{anix.17}, we will employ MMP to derive a contradiction. Suppose $P_d=(0,t_d)$ and $t_d<1$. For any $t\geq0$, denote
\begin{equation*}
w_t(x)=u_d(x)-u_d(x^t),\ \mbox{for}\ x_2\geq t,
\end{equation*}
where $x^t=(x_1,2t-x_2)$. For $t=0$, we have proved $w_0(x)>0$ for $x_2>0$. Define
$$t_0=\sup\{t<t_d|\ \forall x_2\geq s, 0\leq s\leq t, w_s(x)\geq 0\}.$$
Then, we claim that $t_0<t_d$. Since $w_t(x)$ is continuous, we have $w_{t_0}(x)\geq 0$ for $x_2\geq t_0$. Through applying the Hopf lemma to $w_t(x)$ on $x_2=t$, one can derive that for $0\leq t\leq t_0$, $x_2<t$, there holds
$\frac{\partial u_d(x)}{\partial x_2}>0.$
This together with inequality \eqref{anix.17} yields that
\begin{equation}\label{anix.22}
\frac{\partial u_d(x)}{\partial x_2}>0,\ \mbox{for}\ x_2<t_0.
\end{equation}
Moreover, by the strong maximum principle, we know that either $w_{t_0}(x)\equiv0$ or $w_{t_0}(x)>0$ for $x_2> t_0$. If $w_{t_0}(x)\equiv0$ for $x_2> t_0$, then for $x_2> t_0$,
$$\frac{\partial w_{t_0}(x)}{\partial x_2}=0.$$
However, if $x$ is located on the boundary, inequality \eqref{anix.16} and the Neumann condition gives that $\frac{\partial u_d(x)}{\partial x_2}\geq 0$. Combine this with inequality \eqref{anix.22}, one can derive that
$$\frac{\partial w_{t_0}(x)}{\partial x_2}=\frac{\partial u_d(x)}{\partial x_2}+\frac{\partial u_d(x^{t_0})}{\partial x_2}>0,$$
which contradicts with $\frac{\partial w_{t_0}(x)}{\partial x_2}=0$. Therefore $w_{t_0}(x)>0$ for $x_2> t_0$. Thanks to the Hopf lemma again, we can obtain that for $x_2=t_0$,
\begin{equation}\label{anix.24}
2\frac{\partial u_d(x)}{\partial x_2}=\frac{\partial w_{t_0}(x)}{\partial x_2}>0.
\end{equation}
Since $P_d$ is a maximum point, we get $\frac{\partial u_d(x)}{\partial x_2}=0$ for $x_2=t_d$. Therefore $t_0<t_d$.

Once $t_0<t_d$ holds, one can pick $t_j>t_0$ and $t_j\rightarrow t_0$ such that
\begin{equation}\begin{split}\label{anix.23}
w_{t_j}(x_j)=\inf_{x\in \overline{D}\cap\{x_2\geq t_j\}} w_{t_j}(x)<0.
\end{split}\end{equation}
If $x_j$ lies on the boundary, one can obtain that
$$0\geq \frac{\partial w_{t_j}(x_j)}{\partial x_2}=\frac{\partial u_{d}(x_j)}{\partial x_2}+\frac{\partial u_{d}(x^{t_j}_j)}{\partial x_2}\geq \frac{\partial u_{d}(x^{t_j}_j)}{\partial x_2}>0.$$
Therefore, one can get $x_j\in D\cap\{x_n\geq t_j\}$. Up to a sequence, $x_j\rightarrow x_\infty$ as $j\rightarrow+\infty$. By the definition of $x_\infty$, we have $w_{t_0}(x_\infty)=0$ and $\frac{\partial w_{t_0}(x_\infty)}{\partial x_2}=0$. Notice that $w_{t_0}(x)>0$ when $x_2>t_0$, hence $x_\infty\in \{x_2= t_0\}$ and
$$0=\frac{\partial w_{t_0}(x_\infty)}{\partial x_2}=2\frac{\partial u_d(x_\infty)}{\partial x_2},$$
which contradicts with $\frac{\partial u_d(x_\infty)}{\partial x_2}>0$. Therefore, $t_d=1$ and $P_d$ is located on the boundary.

\medskip

{\bf Acknowledgement.} The authors wish to thank A. Malchiodi and J. Wei for comments and for pointing out many relevant references in the literature.

 \end{document}